\documentclass{amsart}

\usepackage{amssymb}
\usepackage[nobysame]{amsrefs}
\usepackage{tikz}
\usepackage{todonotes}
\usepackage{subcaption}
\usepackage{mathpazo}
\usepackage{enumerate}
\usepackage[colorlinks=true,
	urlcolor=blue!50!black,
	citecolor=green!50!black]{hyperref}
\theoremstyle{plain}
\newtheorem{theorem}{Theorem}[section]
\newtheorem{lemma}[theorem]{Lemma}
\newtheorem{proposition}[theorem]{Proposition}
\newtheorem{corollary}[theorem]{Corollary}

\theoremstyle{definition}
\newtheorem{example}[theorem]{Example}
\newtheorem{problem}[theorem]{Problem}
\newcommand{\defn}[1]{{\color{green!50!gray}\emph{#1}}}
\newcommand{\defs}{\stackrel{\mathrm{def}}{=}}
\newcommand{\ie}{\text{i.e.}\;}
\newcommand{\Lattice}{\mathcal{L}}
\newcommand{\Poset}{\mathcal{P}}
\newcommand{\least}{\hat{0}}
\newcommand{\grtst}{\hat{1}}
\newcommand{\ji}{\mathcal{J}}
\newcommand{\mi}{\mathcal{M}}
\newcommand{\Shard}{\Psi}
\newcommand{\Canonical}{\Gamma}

\DeclareMathOperator{\Alternate}{\mathrm CLO}
\DeclareMathOperator{\Con}{\mathrm Con}
\DeclareMathOperator{\cg}{\mathrm cg}
\DeclareMathOperator{\cl}{\mathrm cl}
\DeclareMathOperator{\Boolean}{\mathrm Bool}
\DeclareMathOperator{\Atoms}{\mathrm Atoms}
\DeclareMathOperator{\bdef}{\mathrm bdef}
\DeclareMathOperator{\Cross}{\mathrm Cross}
\DeclareMathOperator{\Nexus}{\mathrm Nexus}
\DeclareMathOperator{\Bic}{\mathrm Bic}

\begin{document}


\title{The Core Label Order of a Congruence-Uniform Lattice}
\author{Henri M{\"u}hle}
\address{Institut f{\"u}r Algebra, Technische Universit{\"a}t Dresden, Zellescher Weg 12--14, 01069 Dresden, Germany.}
\email{henri.muehle@tu-dresden.de}

\begin{abstract}
	We investigate the alternate order on a congruence-uniform lattice $\mathcal{L}$ as introduced by N.~Reading, which we dub the core label order of $\mathcal{L}$.  When $\mathcal{L}$ can be realized as a poset of regions of a simplicial hyperplane arrangement, the core label order is always a lattice.  For general $\mathcal{L}$, however, this fails.  We provide an equivalent characterization for the core label order to be a lattice.  As a consequence we show that the property of the core label order being a lattice is inherited to lattice quotients.  We use the core label order to characterize the congruence-uniform lattices that are Boolean lattices, and we investigate the connection between congruence-uniform lattices whose core label orders are lattices and congruence-uniform lattices of biclosed sets.
\end{abstract}

\maketitle


\section{Introduction}
	\label{sec:introduction}
A (real) hyperplane arrangement $\mathcal{A}$ is a collection of hyperplanes in $\mathbb{R}^{n}$, and the connected components of $\mathbb{R}^{n}\setminus\mathcal{A}$ are called the regions of $\mathcal{A}$.  P.~Edelman defined a partial order on the set of regions of $\mathcal{A}$ with respect to a fixed base region: two regions are comparable in this order whenever we can go from the one region to the other by crossing one hyperplane at a time and never decreasing the number of hyperplanes between the current region and the base region~\cite{edelman84partial}.  

It was shown in \cite{bjorner90hyperplane} that this poset of regions is a lattice whenever $\mathcal{A}$ is simplicial.  Subsequently, N.~Reading thoroughly studied the structure of the poset of regions~\cite{reading03lattice,reading03order}, see also \cite{reading16lattice}.  One of the main results in his study is a characterization of those hyperplane arrangements that have posets of regions which are semidistributive or congruence-uniform lattices, see Theorem~9-3.8 and Corollary~9-7.22 in \cite{reading16lattice}.  See also \cite[Theorem~3]{mcconville17crosscut}.  

A key tool for understanding lattice congruences in the lattice of regions (and therefore congruence-uniformity) are so-called shards of hyperplanes.  The terminology suggests that these can be understood as pieces of hyperplanes that are broken off by intersections with other (in some sense stronger) hyperplanes.  Proposition~3.3 in \cite{reading11noncrossing} states that the shards of a simplicial hyperplane arrangement $\mathcal{A}$ are in bijection with the join-irreducible elements of the lattice of regions (and thus, if this lattice is congruence uniform, with the join-irreducible lattice congruences).  The shards give rise to an alternate partial order on the regions of $\mathcal{A}$: the \defn{shard intersection order}.  It turns out that this order is always a lattice~\cite[Section~4]{reading11noncrossing}.  Perhaps the most prominent example of a shard intersection order is the lattice of noncrossing partitions associated with a finite Coxeter group, which arises from certain quotient lattices of the poset of regions of the corresponding Coxeter arrangement~\cite[Theorem~8.5]{reading11noncrossing}.  These quotient lattices are known as Cambrian lattices; see \cite{reading07clusters,reading07sortable} for more background.  The shard intersection order of the lattice of regions of a Coxeter arrangement was also studied in \cite{bancroft11shard,bancroft11the,petersen13on}.

N.~Reading suggested a generalization of the shard intersection order to arbitrary congruence-uniform lattices~\cite[Section~9-7.4]{reading16lattice}, where he essentially associates a certain set of join-irreducible elements with each lattice element, and then orders these sets by containment.  This article is devoted to the study of this order, which we decided to call the \defn{core label order}; denoted by $\Alternate(\Lattice)$.  The terminology is due to the fact that $\Alternate(\Lattice)$ can be realized as a family of sets of certain edge-labels in the poset diagram of $\Lattice$ ordered by inclusion.  See Section~\ref{sec:alternate_order_definition} for further justification of this terminology.

It turns out that at this level of generality the lattice property of the core label order is no longer guaranteed.  Problem~9.5 in \cite{reading16lattice} asks for conditions on $\Lattice$ such that $\Alternate(\Lattice)$ is again a lattice.  The first main result of this article is a necessary condition stating that if $\Alternate(\Lattice)$ is a lattice, then $\Lattice$ is \defn{spherical}, \ie the order complex of the proper part of $\Lattice$ is homotopic to a sphere.

\begin{theorem}\label{thm:shard_lattice_necessary}
	Let $\Lattice$ be a congruence-uniform lattice.  If $\Alternate(\Lattice)$ is a lattice, then $\Lattice$ is spherical.
\end{theorem}

The proof of Theorem~\ref{thm:shard_lattice_necessary} essentially follows from the semidistributivity of $\Lattice$ and G.-C.~Rota's Crosscut Theorem.  This condition is, however, not sufficient if $\Lattice$ has more than eight elements.  We can explicitly construct spherical congruence-uniform lattices with at least nine elements whose core label order is not a lattice.

\begin{theorem}\label{thm:shard_lattice_broken}
	For all $n\geq 9$ there exists a spherical congruence-uniform lattice of cardinality $n$ whose core label order is not a lattice.
\end{theorem}

The smallest example of the family of congruence-uniform lattices that occur in Theorem~\ref{thm:shard_lattice_broken} is the Boolean lattice of size eight doubled by an atom.  We prove Theorems~\ref{thm:shard_lattice_necessary} and \ref{thm:shard_lattice_broken} in Section~\ref{sec:necessary_condition}, after we have recalled the necessary lattice-theoretic notions in Section~\ref{sec:background}, and have laid some more groundwork in Section~\ref{sec:alternate_order_definition}.

If $\Lattice$ is spherical, then an easy sufficient condition for the lattice property of $\Alternate(\Lattice)$ is that the family of \defn{core label sets} of $\Lattice$ is closed under intersection.  If this is satisfied, we say that $\Lattice$ has the \defn{intersection property}, and we show that this property is in fact equivalent to $\Alternate(\Lattice)$ being a meet-semilattice.  We may conclude the following result.

\begin{theorem}\label{thm:alternate_lattice}
	The core label order of a congruence-uniform lattice $\Lattice$ is a lattice if and only if $\Lattice$ is spherical and has the intersection property.
\end{theorem}

Of course, we may now ask for conditions on $\Lattice$ that ensure the intersection property.  We are able to show that the intersection property is inherited to lattice quotients, which in view of Theorem~\ref{thm:alternate_lattice} implies that the class of congruence-uniform lattices whose core label orders are lattices is closed under taking quotients.

\begin{theorem}\label{thm:sip_congruences}
	Let $\Lattice$ be a spherical congruence-uniform lattice with the intersection property.  For any lattice congruence $\Theta$ of $\Lattice$, the core label order of the quotient lattice $\Lattice/\Theta$ is a lattice.
\end{theorem}

The proofs of Theorems~\ref{thm:alternate_lattice} and \ref{thm:sip_congruences} are both given in Section~\ref{sec:sip}, where we also formally define the intersection property.

In Section~\ref{sec:boolean_nexus} we show that $\Lattice$ and $\Alternate(\Lattice)$ intersect in a common Boolean lattice, which we call the \defn{Boolean nexus} of $\Lattice$.  This lattice is closely related to the crosscut complex of $\Lattice$.  We also show that the Boolean lattices are the only congruence-uniform lattices that are \defn{atomic} (Theorem~\ref{thm:boolean_atomic_semidistributive}), \ie in which every element can be written as a join of atoms.  This observation enables us to prove the following result, which is a new characterization of Boolean lattices in terms of the core label order.

\begin{theorem}\label{thm:boolean_alternate_order}
	Let $\Lattice$ be a congruence-uniform lattice.  We have $\Lattice\cong\Alternate(\Lattice)$ if and only if $\Lattice$ is a Boolean lattice.
\end{theorem}

If we browse the current literature, then it seems that all the available congruence-uniform lattices whose core label orders are lattices have one thing in common: they are (quotients of) lattices of biclosed sets.  Let us postpone the definition until Section~\ref{sec:biclosed}.  

For instance, Theorem~4.2.2 in \cite{mcconville15biclosed} states that every lattice of regions of a simplicial hyperplane arrangement is a lattice of biclosed sets.  Moreover, A.~Garver and T.~McConville define in \cite{garver18oriented} a lattice of biclosed sets of segments of a tree embedded in a disk, and they prove that it is congruence uniform.  Theorems~5.12 and 5.14 in \cite{garver18oriented} imply that the core label order of certain quotients of this lattice has the lattice property; Theorem~5.13 in \cite{clifton18canonical} states that the core label order of the full lattice has the lattice property, too.  In a similar spirit, T.~McConville defines in \cite{mcconville17lattice} a lattice of biclosed sets of segments on a grid, and he proves that it is congruence uniform.  Proposition~5.20 in \cite{garver17enumerative} implies that the core label order of certain quotients of this lattice has the lattice property. 

In Section~\ref{sec:biclosed} we explore the connection between congruence-uniform lattices whose core label orders have the lattice property and congruence-uniform lattices of biclosed sets.  We exhibit a spherical congruence-uniform lattice of biclosed sets whose core label order is not a lattice.  Moreover, in Problem~\ref{prob:single_step_no_lattice} we ask for a graded spherical congruence-uniform lattice of biclosed sets whose core label order is not a lattice.  

\section{Background}
	\label{sec:background}
\subsection{Lattices and Congruences}
	\label{sec:lattices_congruences}
Let $\Lattice=(L,\leq)$ be a \defn{lattice}, \ie a partially ordered set (\defn{poset} for short) in which every two elements $x,y\in L$ have a greatest lower bound (their \defn{meet}; written $x\wedge y$), and a least upper bound (their \defn{join}; written $x\vee y)$.  Throughout the paper we will only consider \underline{finite lattices}.  It follows that $\Lattice$ has a least element $\least$ and a greatest element $\grtst$.  

Two elements $x,y\in L$ form a \defn{cover relation} in $\Lattice$ if $x<y$ and there is no $z\in L$ with $x<z<y$.  We usually write $x\lessdot y$, and say that $x$ is a \defn{lower cover} of $y$; or equivalently that $y$ is an \defn{upper cover} of $x$.  

The \defn{dual} of $\Lattice$ is the lattice $\Lattice^{*}\defs(L,\geq)$.  If $\Lattice\cong\Lattice^{*}$, then $\Lattice$ is \defn{self dual}.

An element $j\in L\setminus\{\least\}$ is \defn{join irreducible} if whenever $j=x\vee y$ for $x,y\in L$, then $j\in\{x,y\}$.  Since $\Lattice$ is finite, it follows that every join-irreducible element $j$ has a unique lower cover $j_{*}$.  Dually, we define the set of \defn{meet-irreducible} elements of $\Lattice$ by $\mi(\Lattice)\defs\ji(\Lattice^{*})$.

A \defn{lattice congruence} is an equivalence relation $\Theta$ on $L$ such that $[x]_{\Theta}=[y]_{\Theta}$ and $[u]_{\Theta}=[v]_{\Theta}$ imply $[x\wedge u]_{\Theta}=[y\wedge v]_{\Theta}$ and $[x\vee u]_{\Theta}=[y\vee v]_{\Theta}$ for all $x,y,u,v\in L$.  The set $\Con(\Lattice)$ of all lattice congruences of $\Lattice$ ordered by refinement is again a lattice~\cite{funayama42on}; the \defn{congruence lattice} of $\Lattice$.  For $x,y\in L$ with $x\lessdot y$, let $\cg(x,y)$ denote the finest lattice congruence of $\Lattice$ in which $x$ and $y$ are equivalent.  If $y\in\ji(\Lattice)$, then we write $\cg(y)$ instead of $\cg(y_{*},y)$.

We have the following characterization of join-irreducible lattice congruences; see \cite[Section~2.14]{gratzer11lattice} for the equivalence of (i) and (ii) and \cite[Theorem~3.20]{freese95free} for the equivalence of (i) and (iii).

\begin{theorem}\label{thm:irreducible_congruences}
	Let $\Lattice$ be a finite lattice, and let $\Theta\in\Con(\Lattice)$.  The following are equivalent.
	\begin{enumerate}[(i)]
		\item $\Theta$ is join-irreducible in $\Con(\Lattice)$.
		\item $\Theta=\cg(x,y)$ for some $x\lessdot y$.
		\item $\Theta=\cg(j)$ for some $j\in\ji(\Lattice)$.
	\end{enumerate}
\end{theorem}

The map $j\mapsto\cg(j)$ is surjective by Theorem~\ref{thm:irreducible_congruences}, but in general it may fail to be injective.  A finite lattice is \defn{congruence uniform} if this map is a bijection for both $\Lattice$ and $\Lattice^{*}$.  

Congruence-uniform lattices sometimes appear in the (universal algebra and lattice theory) literature under the name ``bounded lattices'', which has its origins in \cite{mckenzie72equational} and refers to the fact that these lattices are precisely the bounded-homomorphic images of a free lattice.  This notation, however, clashes with the term ``bounded poset'', which refers simply to the fact that a poset has a least and a greatest element, and is widely used in combinatorics.

\begin{figure}
	\centering
	\begin{subfigure}[t]{.37\textwidth}
		\centering
		\begin{tikzpicture}\small
			\def\x{1};
			\def\y{1};
			\draw(2*\x,1*\y) node(n1){$\least$};
			\draw(1*\x,2*\y) node(n2){$a_{1}$};
			\draw(2*\x,2*\y) node(n3){$a_{2}$};
			\draw(3*\x,2*\y) node(n4){$a_{3}$};
			\draw(2*\x,3*\y) node(n5){$\grtst$};
			\draw(n1) -- (n2);
			\draw(n1) -- (n3);
			\draw(n1) -- (n4);
			\draw(n2) -- (n5);
			\draw(n3) -- (n5);
			\draw(n4) -- (n5);
		\end{tikzpicture}
		\caption{A five-element lattice.}
		\label{fig:m3_lattice}
	\end{subfigure}
	\hspace*{.1cm}
	\begin{subfigure}[t]{.53\textwidth}
		\centering
		\begin{tikzpicture}\small
			\def\x{1};
			\def\y{1};
			\draw(1*\x,1*\y) node(m1){$\bigl\{\{\least\},\{a_{1}\},\{a_{2}\},\{a_{3}\},\{\grtst\}\bigr\}$};
			\draw(1*\x,2.5*\y) node(m2){$\bigl\{\{\least,a_{1},a_{2},a_{3},\grtst\}\bigr\}$};
			\draw(m1) -- (m2);
		\end{tikzpicture}
		\caption{The congruence lattice of the lattice in Figure~\ref{fig:m3_lattice}.}
		\label{fig:m3_congruence}
	\end{subfigure}
	\caption{A lattice that is not congruence uniform.}
	\label{fig:m3}
\end{figure}
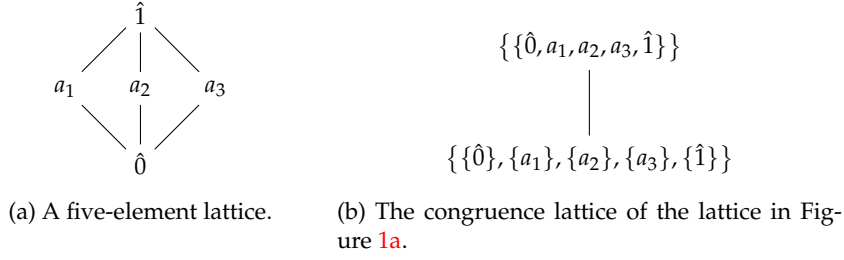

\begin{figure}
	\centering
	\begin{subfigure}[t]{.37\textwidth}
		\centering
		\begin{tikzpicture}\small
			\def\x{1};
			\def\y{1};
			\draw(2*\x,1*\y) node(n1){$\least$};
			\draw(1*\x,2*\y) node(n2){$a_{1}$};
			\draw(3*\x,2.5*\y) node(n3){$a_{2}$};
			\draw(1*\x,3*\y) node(n4){$b_{1}$};
			\draw(2*\x,4*\y) node(n5){$\grtst$};
			\draw(n1) -- (n2);
			\draw(n1) -- (n3);
			\draw(n2) -- (n4);
			\draw(n3) -- (n5);
			\draw(n4) -- (n5);
		\end{tikzpicture}
		\caption{Another five-element lattice.}
		\label{fig:n5_lattice}
	\end{subfigure}
	\hspace*{.1cm}
	\begin{subfigure}[t]{.53\textwidth}
		\centering
		\begin{tikzpicture}\small
			\def\x{1};
			\def\y{1};
			\draw(3*\x,1*\y) node(m1){$\bigl\{\{\least\},\{a_{1}\},\{a_{2}\},\{b_{1}\},\{\grtst\}\bigr\}$};
			\draw(3*\x,2*\y) node(m2){$\bigl\{\{\least\},\{a_{1},b_{1}\},\{a_{2}\},\{\grtst\}\bigr\}$};
			\draw(1*\x,3*\y) node(m3){$\bigl\{\{\least,a_{1},b_{1}\},\{a_{2},\grtst\}\bigr\}$};
			\draw(5*\x,3*\y) node(m4){$\bigl\{\{\least,a_{2}\},\{a_{1},b_{1},\grtst\}\bigr\}$};
			\draw(3*\x,4*\y) node(m5){$\bigl\{\{\least,a_{1},a_{2},b_{1},\grtst\}\bigr\}$};
			\draw(m1) -- (m2);
			\draw(m2) -- (m3);
			\draw(m2) -- (m4);
			\draw(m3) -- (m5);
			\draw(m4) -- (m5);
		\end{tikzpicture}
		\caption{The congruence lattice of the lattice in Figure~\ref{fig:n5_lattice}.}
		\label{fig:n5_congruence}
	\end{subfigure}
	\caption{A congruence-uniform lattice.}
	\label{fig:n5}
\end{figure}
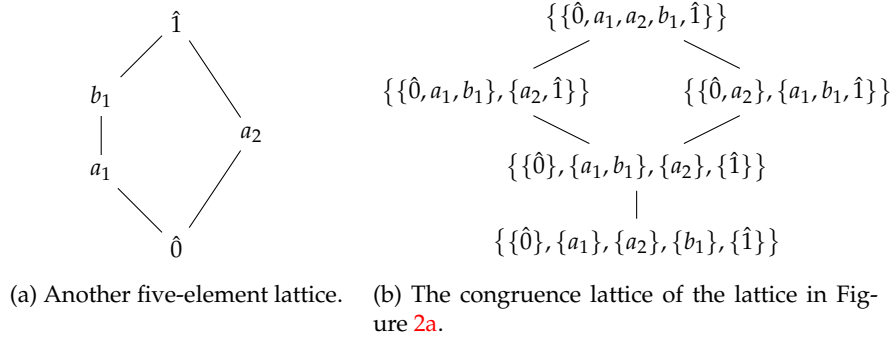

Figures~\ref{fig:m3} and \ref{fig:n5} show two lattices with five elements together with their congruence lattices.  The first example is not congruence uniform, the second one is.  

For later use, let us record that the class of congruence-uniform lattices is a pseudovariety.

\begin{proposition}[{\cite[Theorem~4.3]{day79characterizations}}]\label{prop:congruence_uniform_pseudovariety}
	Congruence-uniformity is preserved under taking quotients, sublattices and finite direct products.
\end{proposition}

\subsection{Doubling by Intervals}
	\label{sec:doubling}
It follows from a result of A.~Day that congruence-uniform lattices can be characterized by means of the following doubling construction.  

Let $\Poset=(P,\leq)$ be an arbitrary (finite) partially ordered set, and let $I\subseteq P$.  Define $P_{\leq I}\defs\{x\in P\mid x\leq y\;\text{for some}\;y\in I\}$.  Let $\mathbf{2}$ denote the $2$-element lattice on the ground set $\{0,1\}$, where we set $0<1$.  The \defn{doubling} of $\Poset$ by $I$ is the subposet of the direct product $\Poset\times\mathbf{2}$ given by the ground set
\begin{displaymath}
	\Bigl(P_{\leq I}\times\{0\}\Bigr)\uplus\Bigl(\bigl((P\setminus P_{\leq I})\cup I\bigr)\times\{1\}\Bigr),
\end{displaymath}
where ``$\uplus$'' denotes disjoint set union.  We denote the resulting poset by $\Poset[I]$, and if $I=\{i\}$ we write $\Poset[i]$ instead of $\Poset\bigl[\{i\}\bigr]$.

If $\Lattice=(L,\leq)$ is a lattice, and $I\subseteq L$ is an arbitrary subset, then it is not necessarily true that $\Lattice[I]$ is still a lattice.  This is illustrated in Figure~\ref{fig:non_lattice_doubling}.  If we restrict ourselves to doublings by order convex subsets, however, then it follows from \cite{day92doubling} that the lattice property is preserved.  (Recall that $I\subseteq P$ is \defn{order convex} if for all $x,y,z\in P$ with $x<y<z$ we have that $x,z\in I$ implies $y\in I$.)

\begin{figure}
	\centering
	\begin{tikzpicture}\small
		\def\x{.75};
		\def\y{.75};
		\def\s{.6};
		\draw(2*\x,1*\y) node[draw,circle,scale=\s](a1){};
		\draw(1*\x,2*\y) node[fill,circle,scale=\s](a2){};
		\draw(3*\x,2*\y) node[fill,circle,scale=\s](a3){};
		\draw(2*\x,3*\y) node[draw,circle,scale=\s](a4){};
		\draw(1*\x,4*\y) node[fill,circle,scale=\s](a5){};
		\draw(3*\x,4*\y) node[fill,circle,scale=\s](a6){};
		\draw(2*\x,5*\y) node[draw,circle,scale=\s](a7){};
		\draw(a1) -- (a2);
		\draw(a1) -- (a3);
		\draw(a2) -- (a4);
		\draw(a3) -- (a4);
		\draw(a4) -- (a5);
		\draw(a4) -- (a6);
		\draw(a5) -- (a7);
		\draw(a6) -- (a7);
		\draw(4*\x,3*\y) node{\tiny $\rightarrow$};
		\draw(6*\x,1*\y) node[draw,circle,scale=\s](b1){};
		\draw(5*\x,2*\y) node[draw,circle,scale=\s](b2){};
		\draw(7*\x,2*\y) node[draw,circle,scale=\s](b3){};
		\draw(6*\x,3*\y) node[draw,circle,scale=\s](b4){};
		\draw(7.5*\x,3*\y) node[draw,circle,scale=\s](b5){};
		\draw(9.5*\x,3*\y) node[draw,circle,scale=\s](b6){};
		\draw(5*\x,4*\y) node[draw,circle,scale=\s](b7){};
		\draw(7*\x,4*\y) node[draw,circle,scale=\s](b8){};
		\draw(7.5*\x,5*\y) node[draw,circle,scale=\s](b9){};
		\draw(9.5*\x,5*\y) node[draw,circle,scale=\s](b10){};
		\draw(8.5*\x,6*\y) node[draw,circle,scale=\s](b11){};
		\draw(b1) -- (b2);
		\draw(b1) -- (b3);
		\draw(b2) -- (b4);
		\draw(b2) -- (b5);
		\draw(b3) -- (b4);
		\draw(b3) -- (b6);
		\draw(b4) -- (b7);
		\draw(b4) -- (b8);
		\draw(b5) -- (b9);
		\draw(b5) -- (b10);
		\draw(b6) -- (b9);
		\draw(b6) -- (b10);
		\draw(b7) -- (b9);
		\draw(b8) -- (b10);
		\draw(b9) -- (b11);
		\draw(b10) -- (b11);
	\end{tikzpicture}
	\caption{If we double the left lattice by the set of solid dots, then we obtain a non-lattice.}
	\label{fig:non_lattice_doubling}
\end{figure}
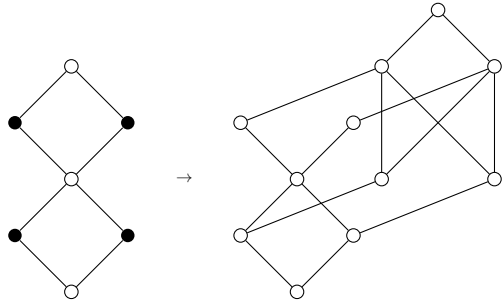

\begin{lemma}\label{lem:doubling_irreducibles}
	Let $\Lattice=(L,\leq)$ be a finite lattice, and let $I\subseteq L$ be order convex.  If the induced subposet $(I,\leq)$ has exactly $k$ minimal elements, then 
	\begin{displaymath}
		\Bigl\lvert\ji\bigl(\Lattice[I]\bigr)\Bigr\rvert = \Bigl\lvert\ji\bigl(\Lattice\bigr)\Bigr\rvert + k.
	\end{displaymath}
\end{lemma}
\begin{proof}
	Since $I$ is order convex it follows from \cite{day92doubling} that $\Lattice[I]$ is again a lattice.  Let $x\in L$.  We count the lower covers of $(x,0)$ resp. $(x,1)$ in $\Lattice[I]$.  
	
	If $x\in L_{\leq I}$, then $(y,a)$ is a lower cover of $(x,0)$ in $\Lattice[I]$ if and only if $a=0$ and $y$ is a lower cover of $x$ in $\Lattice$. 
	
	Now let $x\in L\setminus L_{\leq I}$, and suppose that $(y,a)$ is a lower cover of $(x,1)$ in $\Lattice[I]$.  If $a=0$, then $y\in L_{\leq I}\setminus I$, and $y$ must be a lower cover of $x$ in $\Lattice$.  (If $y\in I$ is a lower cover of $x$ in $\Lattice$ and $a=0$, then we have $(y,0)<(y,1)<(x,1)$ in $\Lattice[I]$, which is a contradiction.)  If $a=1$, then $y\in (L\setminus L_{\leq I})\cup I$, and $y$ must be a lower cover of $x$ in $\Lattice$.  
	
	Finally, let $x\in I$, and suppose that $y_{1},y_{2},\ldots,y_{s}$ are the lower covers of $x$ in the induced subposet $(I,\leq)$.  Then, the lower covers of $(x,1)$ in $\Lattice[I]$ are precisely $(y_{1},1),(y_{2},1),\ldots,(y_{s},1)$ and $(x,0)$.  
\end{proof}

The doubling construction enables us to characterize congruence-uniform lattices in a second way.

\begin{theorem}[{\cite[Theorem~5.1]{day79characterizations}}]\label{thm:congruence_uniform_doubling}
	A finite lattice is congruence uniform if and only if it can be obtained from the singleton lattice by a sequence of doublings by intervals.
\end{theorem}

In particular, Lemma~\ref{lem:doubling_irreducibles} implies that for a congruence-uniform lattice $\Lattice$ the size of $\ji(\Lattice)$ determines the exact number of doubling steps needed to create $\Lattice$.

Figure~\ref{fig:n5_doubling} shows how the lattice in Figure~\ref{fig:n5_lattice} can be obtained by a sequence of three doublings.  The intervals at which we double are marked by solid dots.

\begin{figure}
	\centering
	\begin{tikzpicture}\small
		\def\x{1};
		\def\y{1};
		\def\s{.6};
		\draw(1*\x,1*\y) node[fill,circle,scale=\s](a1){};
		\draw(2*\x,1.5*\y) node{\tiny $\rightarrow$};
		\draw(3*\x,1*\y) node[fill,circle,scale=\s](b1){};
		\draw(3*\x,2*\y) node[fill,circle,scale=\s](b2){};
		\draw(b1) -- (b2);
		\draw(4*\x,1.5*\y) node{\tiny $\rightarrow$};
		\draw(6*\x,1*\y) node[draw,circle,scale=\s](c1){};
		\draw(5*\x,2*\y) node[fill,circle,scale=\s](c2){};
		\draw(7*\x,2*\y) node[draw,circle,scale=\s](c3){};
		\draw(6*\x,3*\y) node[draw,circle,scale=\s](c4){};
		\draw(c1) -- (c2);
		\draw(c1) -- (c3);
		\draw(c2) -- (c4);
		\draw(c3) -- (c4);
		\draw(8*\x,1.5*\y) node{\tiny $\rightarrow$};
		\draw(10*\x,1*\y) node[draw,circle,scale=\s](d1){};
		\draw(9*\x,2*\y) node[draw,circle,scale=\s](d2){};
		\draw(11*\x,2.5*\y) node[draw,circle,scale=\s](d3){};
		\draw(9*\x,3*\y) node[draw,circle,scale=\s](d4){};
		\draw(10*\x,4*\y) node[draw,circle,scale=\s](d5){};
		\draw(d1) -- (d2);
		\draw(d1) -- (d3);
		\draw(d2) -- (d4);
		\draw(d3) -- (d5);
		\draw(d4) -- (d5);
	\end{tikzpicture}
	\caption{The lattice in Figure~\ref{fig:n5_lattice} can be obtained by a sequence of doublings.}
	\label{fig:n5_doubling}
\end{figure}
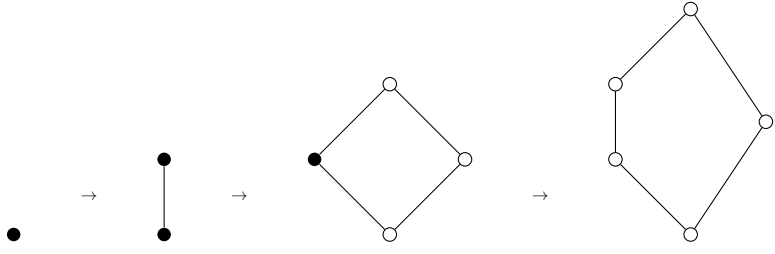

An important family of congruence-uniform lattices are the \defn{Boolean lattices} $\Boolean(M)\defs\bigl(\wp(M),\subseteq\bigr)$, where $\wp(M)$ denotes the power set of $M$.  If $M=[n]\defs\{1,2,\ldots,n\}$, then we usually write $\Boolean(n)$ instead of $\Boolean\bigl([n]\bigr)$.  Observe that we may obtain $\Boolean(n)$ from the singleton lattice by exactly $n$ doublings, where we double at each step by the full lattice.

\subsection{Semidistributive Lattices}
	\label{sec:semidistributive}
A finite lattice $\Lattice=(L,\leq)$ is \defn{join semidistributive} if for every $x,y,z\in L$ with $x\vee y=x\vee z$ we have $x\vee(y\wedge z)=x\vee y$.  We call $\Lattice$ \defn{meet semidistributive} if $\Lattice^{*}$ is join semidistributive, and we say that $\Lattice$ is \defn{semidistributive} if it is both join and meet semidistributive.

\begin{proposition}[{\cite[Lemma~4.2]{day79characterizations}}]\label{prop:congruence_uniform_semidistributive}
	Every congruence-uniform lattice is semidistributive.
\end{proposition}

The converse implication is not true, as is witnessed for instance by the example in Figure~\ref{fig:sd_not_cu}.  See also \cite[Section~3]{nation00unbounded}.  Moreover, \cite[Lemma~2.62]{freese95free} characterizes the semidistributive, congruence-uniform lattices.  

\begin{figure}
	\centering
	\begin{tikzpicture}\small
		\def\x{.75};
		\def\y{.75};
		\def\s{.6};
		\draw(4*\x,1*\y) node[draw,circle,scale=\s](n1){};
		\draw(5*\x,2*\y) node[draw,circle,scale=\s](n2){};
		\draw(2*\x,3*\y) node[draw,circle,scale=\s](n3){};
		\draw(6*\x,3*\y) node[fill,circle,scale=\s](n4){};
		\draw(3*\x,4*\y) node[draw,circle,scale=\s](n5){};
		\draw(5*\x,4*\y) node[fill,circle,scale=\s](n6){};
		\draw(1*\x,4.5*\y) node[draw,circle,scale=\s](n7){};
		\draw(7*\x,4.5*\y) node[draw,circle,scale=\s](n8){};
		\draw(3*\x,5*\y) node[draw,circle,scale=\s](n9){};
		\draw(5*\x,5*\y) node[draw,circle,scale=\s](n10){};
		\draw(2*\x,6*\y) node[draw,circle,scale=\s](n11){};
		\draw(6*\x,6*\y) node[draw,circle,scale=\s](n12){};
		\draw(3*\x,7*\y) node[draw,circle,scale=\s](n13){};
		\draw(4*\x,8*\y) node[draw,circle,scale=\s](n14){};
		\draw(n1) -- (n2);
		\draw(n1) -- (n3);
		\draw(n2) -- (n4);
		\draw(n2) -- (n5);
		\draw(n3) -- (n5);
		\draw(n3) -- (n7);
		\draw(n4) -- (n6);
		\draw(n4) -- (n8);
		\draw(n5) -- (n9);
		\draw(n6) -- (n9);
		\draw(n6) -- (n10);
		\draw(n7) -- (n11);
		\draw(n8) -- (n12);
		\draw(n9) -- (n11);
		\draw(n10) -- (n12);
		\draw(n10) -- (n13);
		\draw(n11) -- (n13);
		\draw(n12) -- (n14);
		\draw(n13) -- (n14);
	\end{tikzpicture}
	\caption{The smallest semidistributive lattice that is not congruence uniform.  The two highlighted join-irreducible elements induce the same lattice congruence.}
	\label{fig:sd_not_cu}
\end{figure}
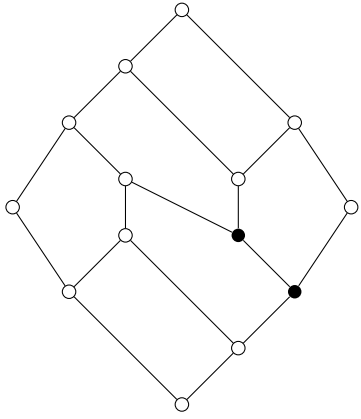

Join-semidistributive lattices have the following characterizing property.  A set $X\subseteq L$ is a \defn{join representation} of $x\in L$ if $\bigvee X=x$.  A join representation $X$ of $x$ is \defn{irredundant} if there is no $X'\subsetneq X$ with $x=\bigvee X'$.  If $X$ and $X'$ are two irredundant join representations of $x$, then $X$ \defn{refines} $X'$ if for every $z\in X$ there exists $z'\in X'$ with $z\leq z'$.  A join representation of $x$ is \defn{canonical} if it is irredundant and refines every other irredundant join representation of $x$; let us write $\Canonical(x)$ for the canonical join representation of $x$ (if it exists).  We define \defn{(canonical) meet representations} dually.

\begin{theorem}[{\cite[Theorem~2.24]{freese95free}}]\label{thm:join_semidistributive_canonical}
	A finite lattice is join semidistributive if and only if every element has a canonical join representation.
\end{theorem}

The next result states that the canonical join representations in fact form a simplical complex; see also \cite{barnard16canonical}.

\begin{proposition}[{\cite[Proposition~2.2]{reading15noncrossing}}]\label{prop:canonical_join_complex}
	Let $(L,\leq)$ be a finite lattice, and let $X\subseteq L$.  If $\bigvee X$ is a canonical join representation, and $X'\subseteq X$, then $\bigvee X'$ is also a canonical join representation.
\end{proposition}

Now suppose that $\Lattice$ is congruence uniform, and pick $x,y\in L$ with $x\lessdot y$.  Theorem~\ref{thm:irreducible_congruences} and the fact that $j\mapsto\cg(j)$ is a bijection imply that there is a unique $j\in\ji(\Lattice)$ with $\cg(j)=\cg(x,y)$; we usually write $j_{\cg(x,y)}$ to denote this element.  This property enables us describe canonical join representations in $\Lattice$ explicitly.

\begin{proposition}[{\cite[Proposition~2.9]{garver18oriented}}]\label{prop:canonical_shards}
	Let $\Lattice=(L,\leq)$ be a congruence uniform lattice.  The canonical join representation of $x\in L$ is 
	\begin{displaymath}
		\Canonical(x) = \bigl\{j_{\cg(y,x)}\mid y\lessdot x\bigr\}.
	\end{displaymath}
\end{proposition}

An \defn{atom} of $\Lattice$ is an element $a\in L$ with $\least\lessdot a$.  Let us write $\Atoms(\Lattice)$ for the set of atoms of $\Lattice$.  A \defn{coatom} of $\Lattice$ is an atom of $\Lattice^{*}$.  

\begin{proposition}\label{prop:atom_join}
	Let $\Lattice$ be a meet semidistributive lattice, and let $X\subseteq\Atoms(\Lattice)$.  If $\bigvee X=\grtst$, then $X=\Atoms(\Lattice)$.
\end{proposition}
\begin{proof}
	Let $X\subsetneq\Atoms(\Lattice)$ with $\bigvee X=\grtst$, and let $a\in \Atoms(\Lattice)\setminus X$.  For every $x\in X$ we have $a\wedge x=\least$, since $a$ and $x$ are atoms.  The meet-semidistributivity of $\Lattice$ then implies that $\least=a\wedge\bigl(\bigvee X\bigr)=a\wedge\grtst=a$.  This contradicts the assumption that $a$ is an atom, and we conclude $X=\Atoms(\Lattice)$.
\end{proof}

A lattice is \defn{atomic} if every element can be expressed as a join of atoms.  We conclude this section with the observation that Boolean lattices are the only atomic semidistributive lattices.

\begin{theorem}\label{thm:boolean_atomic_semidistributive}
	Let $\Lattice$ be a semidistributive lattice.  Then, $\Lattice$ is atomic if and only if $\Lattice\cong\Boolean(n)$ for some $n\in\mathbb{N}$.
\end{theorem}
\begin{proof}
	Let $\Lattice=\Boolean(n)$ for some $n\in\mathbb{N}$, and let $X\subseteq[n]$.  Since 
	\begin{displaymath}
		\Atoms\bigl(\Boolean(n)\bigr)=\bigl\{\{1\},\{2\},\ldots,\{n\}\bigr\},
	\end{displaymath}
	and the join operation in $\Boolean(n)$ is set union, we see directly that $\Boolean(n)$ is atomic.
	
	Conversely, let $\Lattice=(L,\leq)$ be a semidistributive and atomic lattice.  Suppose that $\bigl\lvert\Atoms(\Lattice)\bigr\rvert=n$.  Since $\Lattice$ is semidistributive, we conclude from Theorem~\ref{thm:join_semidistributive_canonical} that every element of $\Lattice$ has a canonical join representation, and since $\Lattice$ is atomic we conclude by definition that any canonical join representation consists of atoms.  It follows that $\lvert L\rvert\leq 2^{n}$.    
	
	Now let $X\subseteq\Atoms(\Lattice)$, and let $x=\bigvee X$.  If $X=\emptyset$, then $x=\least$.  If $X\neq\emptyset$, then $\Canonical(x)\neq\emptyset$.  Suppose that there exists $z\in X\setminus\Canonical(x)$.  For every $a\in\Canonical(x)$ we have $z\wedge a=\least$, since $a$ and $z$ are atoms.  The meet-semidistributivity of $\Lattice$ then implies $\hat{0}=z\wedge\bigl(\bigvee\Canonical(x)\bigr)=z\wedge x=z$, which is a contradiction.  Hence $X=\Canonical(x)$.  We conclude that $\bigvee X$ is a canonical join representation for every $X\subseteq\Atoms(\Lattice)$.  Since $\Lattice$ is atomic we conclude $\lvert L\rvert=2^{n}$.  (The map $x\mapsto\Canonical(x)$ is easily checked to be an isomorphism from $\Lattice$ to $\Boolean\bigl(\Atoms(\Lattice)\bigr)$.)
\end{proof}

\subsection{M{\"o}bius Function and Crosscuts}
	\label{sec:crosscuts}
Let $\Poset=(P,\leq)$ be a finite poset.  The \defn{M{\"o}bius function} of $\Poset$ is the function $\mu_{\Poset}\colon P\times P\to\mathbb{Z}$ defined recursively by
\begin{displaymath}
	\mu_{\Poset}(x,y) \defs \begin{cases}
		1, & \text{if}\;x=y,\\
	        -\sum\limits_{x\leq z<y}{\mu_{\Poset}(x,z)}, & \text{if}\;x<y,\\
	        0 & \text{otherwise}.
	\end{cases}
\end{displaymath}
An \defn{antichain} of $\Poset$ is a subset of $P$ consisting of pairwise incomparable elements.  A \defn{chain} of $\Poset$ is a totally ordered subset of $P$.  A chain is \defn{maximal} if it is maximal under inclusion.

There is a nice combinatorial way to compute the M{\"o}bius function in a finite lattice $\Lattice=(L,\leq)$.  A \defn{crosscut} of $\Lattice$ is an antichain $C\subseteq P$ which contains neither $\least$ nor $\grtst$ and such that any maximal chain of $\Lattice$ intersects $C$ exactly once.  Examples for crosscuts are the sets of atoms or coatoms.  A subset $X\subseteq L$ is \defn{spanning} if $\bigwedge X=\least$ and $\bigvee X=\grtst$.  The following result is known as the Crosscut Theorem.

\begin{theorem}[{\cite[Theorem~3]{rota64foundations}}]\label{thm:crosscut_theorem}
	Let $\Lattice=(L,\leq)$ be a finite lattice and let $C\subseteq L$ be a crosscut.  We have
	\begin{displaymath}
		\mu_{\Lattice}(\least,\grtst) = \sum_{X\subseteq C\;\text{spanning}}{(-1)^{\lvert X\rvert}}.
	\end{displaymath}
\end{theorem}

In the lattice in Figure~\ref{fig:m3_lattice}, the set $C=\{a_{1},a_{2},a_{3}\}$ is a crosscut, and the spanning subsets of $C$ are $C$ itself and every subset of $C$ of size $2$.  We conclude that $\mu(\least,\grtst)=2$, which can also be verified by hand.

We obtain the following result, which may also be concluded from \cite[Theorems~5.1.3~and~5.4.1]{mcconville15biclosed}.

\begin{theorem}\label{thm:meet_semidistributive_mobius}
	If $\Lattice$ is a meet-semidistributive lattice, then $\mu_{\Lattice}(\least,\grtst)\in\{-1,0,1\}$.
\end{theorem}
\begin{proof}
	This follows from Proposition~\ref{prop:atom_join} and Theorem~\ref{thm:crosscut_theorem}.
\end{proof}

Recall that the \defn{order complex} of a finite poset $\Poset$ is the simplicial complex whose faces correspond to the chains of $\Poset$.  If $\Poset=(P,\leq)$ is \defn{bounded} (\ie if it has a least element $\least$ and a greatest element $\grtst$) then we call the poset $\overline{\Poset}\defs\bigl(P\setminus\{\least,\grtst\},\leq\bigr)$ the \defn{proper part} of $\Poset$.  A bounded poset is \defn{spherical} if the order complex of its proper part is homotopy equivalent to a sphere.  

A famous result of P.~Hall states that $\mu_{\Poset}(\least,\grtst)$ equals the reduced Euler characteristic of the order complex of $\overline{\Poset}$; see~\cite[Proposition~3.8.6]{stanley01enumerative}.  It thus follows from Theorem~\ref{thm:meet_semidistributive_mobius} that a meet-semidistributive lattice $\Lattice$ is spherical if and only if $\mu_{\Lattice}(\least,\grtst)\neq 0$.  

We have the following characterization of spherical meet-semidistributive lattices.

\begin{proposition}\label{prop:atom_join_spherical}
	In a meet-semidistributive lattice $\Lattice$ we have $\bigvee\Atoms(\Lattice)=\grtst$ if and only if $\mu_{\Lattice}(\least,\grtst)\neq 0$.
\end{proposition}
\begin{proof}
	Let $e$ (resp. $o$) denote the number of spanning subsets of $\Atoms(\Lattice)$ of even (resp. odd) size.  Proposition~\ref{prop:atom_join} implies that $e+o\leq 1$.

	If $\mu_{\Lattice}(\least,\grtst)=0$, then Theorem~\ref{thm:crosscut_theorem} implies that $e=o$, which forces $e=o=0$.  Hence $\bigvee\Atoms(\Lattice)<\grtst$.  
	
	Conversely if $\mu_{\Lattice}(\least,\grtst)\neq 0$, then Theorem~\ref{thm:meet_semidistributive_mobius} implies $\mu_{\Lattice}(\least,\grtst)=\pm 1$.  Hence we have either $e=1$ and $o=0$, or $e=0$ and $o=1$.  Proposition~\ref{prop:atom_join} implies that $\bigvee\Atoms(\Lattice)=\grtst$.
\end{proof}

\section{The Core Label Order of a Congruence-Uniform Lattice}
	\label{sec:alternate_order}
\subsection{The Core Label Order}
	\label{sec:alternate_order_definition}
Let $\Lattice=(L,\leq)$ be a congruence uniform lattice.  N.~Reading defined in \cite[Section~9-7.4]{reading16lattice} an alternate partial order on $L$ as follows.  The \defn{nucleus} of $x\in L$ is 
\begin{displaymath}
	x_{\downarrow} \defs \bigwedge_{y\in L:y\lessdot x}{y}.
\end{displaymath}
The terminology is due to the fact that the interval $[x_{\downarrow},x]$ is a \defn{nuclear interval}, \ie an interval in which the top element is the join of all upper covers of the bottom element.  Moreover, we call the interval $[x_{\downarrow},x]$ the \defn{core} of $x$.  We also define 
\begin{displaymath}
	\Shard_{\Lattice}(x) \defs \left\{j_{\cg(u,v)}\mid x_{\downarrow}\leq u\lessdot v\leq x\right\}.
\end{displaymath}
In other words, $\Shard_{\Lattice}(x)$ is the set of ``labels'' of the core of $x$, or simply the \defn{core label set} of $x$.  (Observe that we may label every cover relation $u\lessdot v$ in a congruence-uniform lattice by the join-irreducible element $j_{\cg(u,v)}$.)  We omit the subscript $\Lattice$ whenever no confusion can arise.  

There is an easy way to obtain $j_{\cg(u,v)}$ from the cover relation $u\lessdot v$ without having to compare congruences.  Recall that two cover relations $x_{1}\lessdot y_{1}$ and $x_{2}\lessdot y_{2}$ are \defn{perspective} if either $y_{1}\vee x_{2}=y_{2}$ and $y_{1}\wedge x_{2}=x_{1}$ or $y_{2}\vee x_{1}=y_{1}$ and $y_{2}\wedge x_{1}=x_{2}$.  The next result implies that $\Psi(x)$ contains precisely those join-irreducible elements $j$ such that $(j_{*},j)$ is perspective to some cover relation in $[x_{\downarrow},x]$.

\begin{lemma}[{\cite[Lemma~2.6]{garver18oriented}}]\label{lem:perspective_labels}
	Let $\Lattice=(L,\leq)$ be a congruence-uniform lattice, and let $u,v\in L$ such that $u\lessdot v$.  For $j\in\ji(\Lattice)$ holds $\cg(j)=\cg(u,v)$ if and only if $j_{*}\lessdot j$ and $u\lessdot v$ are perspective.
\end{lemma}

For $x,y\in L$ we define $x\sqsubseteq y$ if and only if $\Shard(x)\subseteq\Shard(y)$, and we call the poset 
\begin{displaymath}
	\Alternate(\Lattice) \defs (L,\sqsubseteq)
\end{displaymath}
the \defn{core label order} of $\Lattice$.  The assignment $x\mapsto\Psi_{\Lattice}(x)$ is injective by virtue of Theorem~\ref{thm:irreducible_congruences} and the fact that the map $j\mapsto\cg(j)$ is a bijection.  Therefore, the relation $\sqsubseteq$ is indeed a partial order.

The main motivation for this definition comes from the poset of regions in a hyperplane arrangement.  We have the following result.

\begin{theorem}[{\cite[Section~9-7.4]{reading16lattice}}]\label{thm:region_poset_shard_lattice}
	Let $\Lattice$ be a poset of regions of a hyperplane arrangement.  If $\Lattice$ is a congruence-uniform lattice, then $\Alternate(\Lattice)$ is a lattice.
\end{theorem}

The hyperplane arrangements that have posets of regions which are congruence-uniform lattices are characterized in \cite[Corollary~9-7.22]{reading16lattice}.  In the case described in Theorem~\ref{thm:region_poset_shard_lattice}, the poset $\Alternate(\Lattice)$ is usually referred to as the \defn{shard intersection order}.  However, in the general setting, the term ``shard'' is not really justified, and the collection $\bigl\{\Shard(x)\mid x\in L\bigr\}$ is in general not closed under intersections.

\begin{figure}
	\begin{subfigure}[t]{.45\textwidth}
		\centering
		\begin{tikzpicture}\small
			\def\x{1.5};
			\def\y{1};
			\draw(2*\x,1*\y) node(n1){$\least$};
			\draw(1.5*\x,2*\y) node(n2){$a_{1}$};
			\draw(2.5*\x,2*\y) node(n3){$a_{2}$};
			\draw(1*\x,3*\y) node(n4){$b_{1}$};
			\draw(2*\x,3*\y) node(n5){$b_{2}$};
			\draw(3*\x,3*\y) node(n6){$b_{3}$};
			\draw(1*\x,4*\y) node(n7){$c_{1}$};
			\draw(2*\x,4*\y) node(n8){$c_{2}$};
			\draw(3*\x,4*\y) node(n9){$c_{3}$};
			\draw(1.5*\x,5*\y) node(n10){$d_{1}$};
			\draw(2.5*\x,5*\y) node(n11){$d_{2}$};
			\draw(2*\x,6*\y) node(n12){$\grtst$};
			\draw(n1) -- (n2) node at(1.75*\x,1.5*\y) [fill=white,inner sep=.9pt] {\tiny\color{white!50!black}$\mathbf{1}$};
			\draw(n1) -- (n3) node at(2.25*\x,1.5*\y) [fill=white,inner sep=.9pt] {\tiny\color{white!50!black}$\mathbf{2}$};
			\draw(n2) -- (n4) node at(1.25*\x,2.5*\y) [fill=white,inner sep=.9pt] {\tiny\color{white!50!black}$\mathbf{3}$};
			\draw(n2) -- (n5) node at(1.75*\x,2.5*\y) [fill=white,inner sep=.9pt] {\tiny\color{white!50!black}$\mathbf{2}$};
			\draw(n3) -- (n5) node at(2.25*\x,2.5*\y) [fill=white,inner sep=.9pt] {\tiny\color{white!50!black}$\mathbf{1}$};
			\draw(n3) -- (n6) node at(2.75*\x,2.5*\y) [fill=white,inner sep=.9pt] {\tiny\color{white!50!black}$\mathbf{4}$};
			\draw(n4) -- (n7) node at(1*\x,3.5*\y) [fill=white,inner sep=.9pt] {\tiny\color{white!50!black}$\mathbf{5}$};
			\draw(n5) -- (n8) node at(2*\x,3.5*\y) [fill=white,inner sep=.9pt] {\tiny\color{white!50!black}$\mathbf{6}$};
			\draw(n6) -- (n9) node at(3*\x,3.5*\y) [fill=white,inner sep=.9pt] {\tiny\color{white!50!black}$\mathbf{7}$};
			\draw(n7) -- (n10) node at(1.25*\x,4.5*\y) [fill=white,inner sep=.9pt] {\tiny\color{white!50!black}$\mathbf{2}$};
			\draw(n8) -- (n10) node at(1.75*\x,4.5*\y) [fill=white,inner sep=.9pt] {\tiny\color{white!50!black}$\mathbf{3}$};
			\draw(n8) -- (n11) node at(2.25*\x,4.5*\y) [fill=white,inner sep=.9pt] {\tiny\color{white!50!black}$\mathbf{4}$};
			\draw(n9) -- (n11) node at(2.75*\x,4.5*\y) [fill=white,inner sep=.9pt] {\tiny\color{white!50!black}$\mathbf{1}$};
			\draw(n10) -- (n12) node at(1.75*\x,5.5*\y) [fill=white,inner sep=.9pt] {\tiny\color{white!50!black}$\mathbf{4}$};
			\draw(n11) -- (n12) node at(2.25*\x,5.5*\y) [fill=white,inner sep=.9pt] {\tiny\color{white!50!black}$\mathbf{3}$};
		\end{tikzpicture}
		\caption{A congruence-uniform lattice.}
		\label{fig:cu_lattice}
	\end{subfigure}
	\hspace*{.5cm}
	\begin{subfigure}[t]{.45\textwidth}
		\centering
		\begin{tikzpicture}\small
			\def\x{.75};
			\def\y{1.5};
			\draw(4*\x,1*\y) node(n1){$\least$};
			\draw(2*\x,2*\y) node(n2){$a_{1}$};
			\draw(1*\x,2*\y) node(n3){$a_{2}$};
			\draw(3*\x,2*\y) node(n4){$c_{1}$};
			\draw(4*\x,2*\y) node(n5){$c_{2}$};
			\draw(5*\x,2*\y) node(n6){$c_{3}$};
			\draw(6*\x,2*\y) node(n7){$b_{1}$};
			\draw(7*\x,2*\y) node(n8){$b_{3}$};
			\draw(1.5*\x,3*\y) node(n9){$b_{2}$};
			\draw(3*\x,3*\y) node(n10){$d_{1}$};
			\draw(5*\x,3*\y) node(n11){$d_{2}$};
			\draw(6.5*\x,3*\y) node(n12){$\grtst$};
			\draw(n1) -- (n2);
			\draw(n1) -- (n3);
			\draw(n1) -- (n4);
			\draw(n1) -- (n5);
			\draw(n1) -- (n6);
			\draw(n1) -- (n7);
			\draw(n1) -- (n8);
			\draw(n2) -- (n9);
			\draw(n2) -- (n11);
			\draw(n3) -- (n9);
			\draw(n3) -- (n10);
			\draw(n4) -- (n10);
			\draw(n5) -- (n10);
			\draw(n5) -- (n11);
			\draw(n6) -- (n11);
			\draw(n7) -- (n10);
			\draw(n7) -- (n12);
			\draw(n8) -- (n11);
			\draw(n8) -- (n12);
		\end{tikzpicture}
		\caption{The core label order of the lattice in Figure~\ref{fig:cu_lattice}.}
		\label{fig:cu_lattice_shards}
	\end{subfigure}
	\caption{A congruence-uniform lattice whose core label order is not a lattice.}
	\label{fig:cu}
\end{figure}
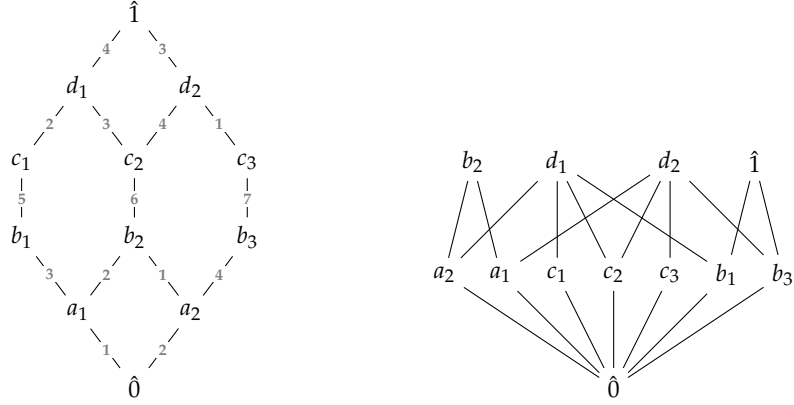

If $\Lattice$ is a congruence-uniform lattice that does not arise as a poset of regions of some hyperplane arrangement, then the core label order need not be a lattice.  Consider for instance the lattice $\Lattice$ in Figure~\ref{fig:cu_lattice}.  The cover relations of $\Lattice$ are labeled so that they reflect the linear extension
\begin{displaymath}
	a_{1}\prec a_{2}\prec b_{1}\prec b_{3}\prec c_{1}\prec c_{2}\prec c_{3}
\end{displaymath}
of the poset of join-irreducible elements of $\Lattice$.  Observe that this labeling also indicates how $\Lattice$ arises from doublings by intervals: the $i^{\mathrm{th}}$ join-irreducible element in this linear extension is created in the $i^{\mathrm{th}}$ doubling step.  We thus obtain
\begin{displaymath}\begin{aligned}
	& \Shard(\least)=\emptyset, && \Shard(a_{1})=\{1\}, && \Shard(a_{2})=\{2\}, && \Shard(b_{1})=\{3\},\\
	& \Shard(b_{2})=\{1,2\}, && \Shard(b_{3})=\{4\}, && \Shard(c_{1})=\{5\}, && \Shard(c_{2})=\{6\},\\
	& \Shard(c_{3})=\{7\}, && \Shard(d_{1})=\{2,3,5,6\}, && \Shard(d_{2})=\{1,4,6,7\}, && \Shard(\grtst)=\{3,4\}.
\end{aligned}\end{displaymath}
Figure~\ref{fig:cu_lattice_shards} shows the core label order of $\Lattice$.  We observe that this poset is a meet-semilattice, \ie any two elements have a meet, but it is not a lattice since it does not have a greatest element.  We observe further that $\Lattice$ is not spherical.  

Let us close this section with the following result, which implies that there exists a bijection on the elements of a congruence-uniform lattice that swaps canonical join representations with canonical meet representations.

\begin{lemma}[{\cite[Lemma~2.12]{garver18oriented}}]\label{lem:representation_swap}
	Let $\Lattice=(L,\leq)$ be a congruence-uniform lattice, and let $y\in L$.  Let $a_{1},a_{2},\ldots,a_{s}$ be some elements that cover $y$, and let $x=a_{1}\vee a_{2}\vee\cdots\vee a_{s}$.  Then there exist elements $c_{1},c_{2},\ldots,c_{s}$ that that are covered by $x$ such that $y=c_{1}\wedge c_{2}\wedge\cdots\wedge c_{s}$ and $\cg(y,a_{i})=\cg(c_{i},x)$ for all $i\in[s]$.  
\end{lemma}

In particular, if $x$ and $y$ are as in Lemma~\ref{lem:representation_swap}, then we have $x_{\downarrow}\leq y$.

\section{Conditions for the Lattice Property of $\Alternate(\Lattice)$}
	\label{sec:conditions}
\subsection{A Necessary Condition for the Lattice Property of $\Alternate(\Lattice)$}
	\label{sec:necessary_condition}
Our first main result, Theorem~\ref{thm:shard_lattice_necessary}, which we are going to prove in the remainder of this section, establishes that the core label order of $\Lattice$ is a lattice only if $\Lattice$ is spherical.

\begin{lemma}\label{lem:shards_irreducibles}
	If $j\in\Shard(x)$, then $j\leq x$.
\end{lemma}
\begin{proof}
	By definition, $j\in\Shard(x)$ means that $j=j_{\cg(u,v)}$ for some $x_{\downarrow}\leq u\lessdot v\leq x$.  In particular, $\cg(u,v)=\cg(j)$, so that  Lemma~\ref{lem:perspective_labels} implies that $j\vee u=v$, and thus $j\leq v\leq x$.
\end{proof}

\begin{lemma}\label{lem:all_shards}
	We have $x=\bigvee\Shard(x)$.
\end{lemma}
\begin{proof}
	By Proposition~\ref{prop:canonical_shards}, we have $\Canonical(x)\subseteq\Shard(x)$, and by definition follows $x=\bigvee\Canonical(x)$.  For $j\in\Shard(x)\setminus\Canonical(x)$ we conclude from Lemma~\ref{lem:shards_irreducibles} that $j\leq x$, and therefore $x=x\vee j$.  This yields the claim.
\end{proof}

\begin{corollary}\label{cor:shard_order_weakening}
	If $\Shard(x)\subseteq\Shard(y)$, then $x\leq y$.
\end{corollary}
\begin{proof}
	This is a direct computation using Lemma~\ref{lem:all_shards}.
\end{proof}

\begin{corollary}\label{cor:maximal_shard_element}
	The greatest element of $\Lattice$ is maximal in $\Alternate(\Lattice)$. 
\end{corollary}
\begin{proof}
	Let $\grtst$ denote the greatest element of $\Lattice$.  If there is $x\in L$ with $\Shard(\grtst)\subseteq\Shard(x)$, then by Corollary~\ref{cor:shard_order_weakening} we conclude $\grtst\leq x$, which implies $x=\grtst$, since $\grtst$ is maximal in $\Lattice$.
\end{proof}

\begin{lemma}\label{lem:full_shards}
	We have $\Shard(\grtst)=\ji(\Lattice)$ if and only if $\mu_{\Lattice}(\least,\grtst)\neq 0$.
\end{lemma}
\begin{proof}
	Let $C$ denote the set of coatoms of $\Lattice$.  
	
	If $\mu_{\Lattice}(\least,\grtst)\neq 0$, then the dual of Proposition~\ref{prop:atom_join_spherical} implies that $\bigwedge C=\least$, so that by definition $\Shard(\grtst)$ contains all join-irreducible elements of $\Lattice$.
	
	If $\mu_{\Lattice}(\least,\grtst)=0$, then the dual of Proposition~\ref{prop:atom_join_spherical} implies that $\bigwedge C=x>\least$.  In particular, there is some $a\in\Atoms(\Lattice)$ with $a\leq x$.  If $a\in\Shard(\grtst)$, then there exist $u,v\in L$ with $x\leq u\lessdot v$ such that $\cg(u,v)=\cg(a)$.  Lemma~\ref{lem:perspective_labels} implies $a\vee u=v$.  However, $a\leq x\leq u$ implies $a\vee u=u$, which is a contradiction.  We conclude that $a\notin\Shard(\grtst)$.
\end{proof}

\begin{lemma}\label{lem:maximal_shard_spherical}
	There exists a greatest element in $\Alternate(\Lattice)$ if and only if $\mu_{\Lattice}(\least,\grtst)\neq 0$.
\end{lemma}
\begin{proof}
	Let $C$ denote the set of coatoms of $\Lattice$.  
	
	If $\mu_{\Lattice}(\least,\grtst)\neq 0$, then Lemma~\ref{lem:full_shards} implies $\Shard(\grtst)=\ji(\Lattice)$.  It follows that for any $x\in L$ we have $\Shard(x)\subseteq\Shard(\grtst)$, which implies that $\grtst$ is the unique maximal element of $\Alternate(\Lattice)$.
	
	If $\mu_{\Lattice}(\least,\grtst)=0$, then Lemma~\ref{lem:full_shards} implies that there is $j\in\ji(\Lattice)$ with $j\notin\Shard(\grtst)$. Since $\grtst$ is maximal in $\Alternate(\Lattice)$ by Corollary~\ref{cor:maximal_shard_element}, we conclude that it is incomparable to $j$ in $\Alternate(\Lattice)$.  The maximality of $\grtst$ implies further that there is no upper bound for $\grtst$ and $j$ in $\Alternate(\Lattice)$, which therefore does not have a greatest element.
\end{proof}

We can now conclude the proof of Theorem~\ref{thm:shard_lattice_necessary}.

\begin{proof}[Proof of Theorem~\ref{thm:shard_lattice_necessary}]
	If $\mu_{\Lattice}(\least,\grtst)=0$, then Lemma~\ref{lem:maximal_shard_spherical} implies that $\Alternate(\Lattice)$ does not have a greatest element.  Since $\Lattice$ is finite, $\Alternate(\Lattice)$ can therefore not be a lattice.
\end{proof}

The example in Figure~\ref{fig:congruence_uniform_no_shard_lattice} illustrates that there exist spherical congruence-uniform lattices whose core label order is not a lattice.  The labels reflect the linear extension $a_{1}\prec a_{2}\prec a_{3}\prec b_{1}$ of the poset of join-irreducible elements of this lattice.

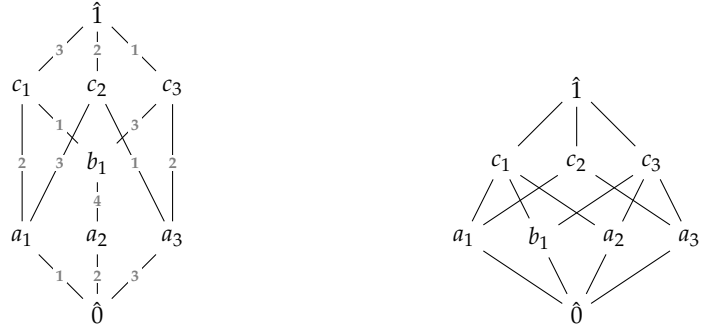
\begin{figure}
	\centering
	\begin{subfigure}[t]{.45\textwidth}
		\centering
		\begin{tikzpicture}\small
			\def\x{1};
			\def\y{1};
			\draw(2*\x,1*\y) node(n1){$\least$};
			\draw(1*\x,2*\y) node(n2){$a_{1}$};
			\draw(2*\x,2*\y) node(n3){$a_{2}$};
			\draw(3*\x,2*\y) node(n4){$a_{3}$};
			\draw(2*\x,3*\y) node(n5){$b_{1}$};
			\draw(1*\x,4*\y) node(n6){$c_{1}$};
			\draw(2*\x,4*\y) node(n7){$c_{2}$};
			\draw(3*\x,4*\y) node(n8){$c_{3}$};
			\draw(2*\x,5*\y) node(n9){$\grtst$};
			\draw(n1) -- (n2) node[fill=white,inner sep=.9pt] at (1.5*\x,1.5*\y){\tiny\color{white!50!black}$\mathbf{1}$};
			\draw(n1) -- (n3) node[fill=white,inner sep=.9pt] at (2*\x,1.5*\y){\tiny\color{white!50!black}$\mathbf{2}$};
			\draw(n1) -- (n4) node[fill=white,inner sep=.9pt] at (2.5*\x,1.5*\y){\tiny\color{white!50!black}$\mathbf{3}$};
			\draw(n2) -- (n6) node[fill=white,inner sep=.9pt] at (1*\x,3*\y){\tiny\color{white!50!black}$\mathbf{2}$};
			\draw(n2) -- (n7) node[fill=white,inner sep=.9pt] at (1.5*\x,3*\y){\tiny\color{white!50!black}$\mathbf{3}$};
			\draw(n3) -- (n5) node[fill=white,inner sep=.9pt] at (2*\x,2.5*\y){\tiny\color{white!50!black}$\mathbf{4}$};
			\draw(n4) -- (n7) node[fill=white,inner sep=.9pt] at (2.5*\x,3*\y){\tiny\color{white!50!black}$\mathbf{1}$};
			\draw(n4) -- (n8) node[fill=white,inner sep=.9pt] at (3*\x,3*\y){\tiny\color{white!50!black}$\mathbf{2}$};
			\draw(n5) -- (n6) node[fill=white,inner sep=.9pt] at (1.5*\x,3.5*\y){\tiny\color{white!50!black}$\mathbf{1}$};
			\draw(n5) -- (n8) node[fill=white,inner sep=.9pt] at (2.5*\x,3.5*\y){\tiny\color{white!50!black}$\mathbf{3}$};
			\draw(n6) -- (n9) node[fill=white,inner sep=.9pt] at (1.5*\x,4.5*\y){\tiny\color{white!50!black}$\mathbf{3}$};
			\draw(n7) -- (n9) node[fill=white,inner sep=.9pt] at (2*\x,4.5*\y){\tiny\color{white!50!black}$\mathbf{2}$};
			\draw(n8) -- (n9) node[fill=white,inner sep=.9pt] at (2.5*\x,4.5*\y){\tiny\color{white!50!black}$\mathbf{1}$};
		\end{tikzpicture}
		\caption{A spherical congruence-uniform lattice.}
		\label{fig:congruence_uniform_no_shard_lattice_cu}
	\end{subfigure}
	\hspace*{.5cm}
	\begin{subfigure}[t]{.45\textwidth}
		\centering
		\begin{tikzpicture}\small
			\def\x{1};
			\def\y{1};
			\draw(3*\x,1*\y) node(m1){$\least$};
			\draw(1.5*\x,2*\y) node(m2){$a_{1}$};
			\draw(2.5*\x,2*\y) node(m3){$b_{1}$};
			\draw(3.5*\x,2*\y) node(m4){$a_{2}$};
			\draw(4.5*\x,2*\y) node(m5){$a_{3}$};
			\draw(2*\x,3*\y) node(m6){$c_{1}$};
			\draw(3*\x,3*\y) node(m7){$c_{2}$};
			\draw(4*\x,3*\y) node(m8){$c_{3}$};
			\draw(3*\x,4*\y) node(m9){$\grtst$};
			\draw(m1) -- (m2);
			\draw(m1) -- (m3);
			\draw(m1) -- (m4);
			\draw(m1) -- (m5);
			\draw(m2) -- (m6);
			\draw(m2) -- (m7);
			\draw(m3) -- (m6);
			\draw(m3) -- (m8);
			\draw(m4) -- (m6);
			\draw(m4) -- (m8);
			\draw(m5) -- (m7);
			\draw(m5) -- (m8);
			\draw(m6) -- (m9);
			\draw(m7) -- (m9);
			\draw(m8) -- (m9);
		\end{tikzpicture}
		\caption{The core label order of the lattice in Figure~\ref{fig:congruence_uniform_no_shard_lattice_cu}.}
		\label{fig:congruence_uniform_no_shard_lattice_shard}
	\end{subfigure}
	\caption{A spherical congruence-uniform lattice whose core label order is not a lattice.}
	\label{fig:congruence_uniform_no_shard_lattice}
\end{figure}

Observe that Figure~\ref{fig:congruence_uniform_no_shard_lattice_cu} is isomorphic to $\Boolean(3)$ doubled by an atom, and it is exactly this doubling that kills the lattice property of the core label order.  We conclude the following result.

\begin{proposition}\label{prop:doubling_no_lattice}
	Let $x,y\in L$ be such that $\Shard(j)\subseteq\Shard(x)\cap\Shard(y)$ for some $j\in\ji(\Lattice)$ which satisfies $j\in\bigl[x_{\downarrow},x\bigr]\cap\bigl[y_{\downarrow},y\bigr]$.  The core label order of $\Lattice[j]$ is not a lattice.
\end{proposition}
\begin{proof}
	We identify the element $(z,i)\in L[j]$ with $z$, except for the case $z=j$ and $i=1$; in this case we denote the element $(j,1)$ by $j'$.  We conclude from Lemma~\ref{lem:doubling_irreducibles} that $j'\in\ji\bigl(\Lattice[j]\bigr)$.  Moreover, we have $j<j'$ in $\Lattice[j]$.
	
	Corollary~\ref{cor:shard_order_weakening} implies that $j$ is a lower bound for $x$ and $y$ in $\Lattice$, and it follows by construction that $j'$ is a lower bound for $x$ and $y$ in $\Lattice[j]$.  We thus have $\Shard_{\Lattice[j]}(x)=\Shard_{\Lattice}(x)\cup\{j'\}$ and $\Shard_{\Lattice[j]}(y)=\Shard_{\Lattice}(y)\cup\{j'\}$.  By assumption we have $\{j\}=\Shard_{\Lattice[j]}(j)\subseteq\Shard_{\Lattice[j]}(x)\cap\Shard_{\Lattice[j]}(y)$ and by construction follows $\{j'\}=\Shard_{\Lattice[j]}(j')\subseteq\Shard_{\Lattice[j]}(x)\cap\Shard_{\Lattice[j]}(y)$.  We conclude that $\Alternate(\Lattice)$ is not a lattice.
\end{proof}

We certainly cannot leave out the extra condition on $j$ in Proposition~\ref{prop:doubling_no_lattice}, since we need to double at an interval contained in $\bigl[x_{\downarrow},x\bigr]\cap\bigl[y_{\downarrow},y\bigr]$ in order to ensure that $\Shard_{\Lattice}(x)\neq\Shard_{\Lattice[j]}(x)$ and $\Shard_{\Lattice}(y)\neq\Shard_{\Lattice[j]}(y)$.  We may now prove Theorem~\ref{thm:shard_lattice_broken}.

\begin{proof}[Proof of Theorem~\ref{thm:shard_lattice_broken}]
	Let $\Lattice$ be a congruence-uniform lattice with at least three atoms $a,b,c$.  Let $x=a\vee b$ and $y=b\vee c$.  Since $\Lattice$ is join semidistributive by Proposition~\ref{prop:congruence_uniform_semidistributive}, we conclude that $x\neq y$ as well as $\Canonical(x)=\{a,b\}$ and $\Canonical(y)=\{b,c\}$.
	
	Proposition~\ref{prop:canonical_shards} implies that there are exactly two lower covers of $x$, say $r_{1}$ and $r_{2}$, and let $r=r_{1}\wedge r_{2}$.  Since $r<x=a\vee b$, we conclude that $a\not\leq r$.  (Analogously follows $b\not\leq r$.)  Since $a$ and $b$ are atoms we conclude that $a\wedge r=\least=b\wedge r$.  The meet-semidistributivity of $\Lattice$ implies that $\least=(a\vee b)\wedge r=x\wedge r=r$.  We conclude that $\Shard(b)\subseteq\Shard(x)$.  By symmetry we obtain $\Shard(b)\subseteq\Shard(y)$.  
	
	Since we have just seen that $x_{\downarrow}=\least=y_{\downarrow}$, and since $b\leq x$ and $b\leq y$ by construction, we conclude that $b\in[\least,x]\cap[\least,y]$.  We can therefore apply Proposition~\ref{prop:doubling_no_lattice}, which proves that $\Alternate\bigl(\Lattice[b]\bigr)$ is not a lattice.  Moreover, we have $\mu_{\Lattice[b]}(\least,\grtst)=\mu_{\Lattice}(\least,\grtst)$, so that $\Lattice[b]$ is spherical if and only if $\Lattice$ is.  
	
	The smallest congruence-uniform lattice with three atoms is $\Boolean(3)$, which implies that the smallest example of the previously described construction has nine elements.
\end{proof}

It can be quickly verified that sphericity of $\Lattice$ is a sufficient condition for the lattice property of $\Alternate(\Lattice)$ if $\Lattice$ has at most eight elements.  The example in Figure~\ref{fig:congruence_uniform_no_shard_lattice} is thus the smallest spherical congruence-uniform lattice whose core label order is not a lattice.  Table~\ref{tab:enumeration} lists the number of congruence-uniform lattices of size $\leq 14$, and the number of such lattices that are spherical and have an core label order that is a lattice.  These numbers were obtained with the help of \texttt{Sage-Combinat}~\cite{sagecombinat,sage}.  Let us use the following abbreviations:
\begin{itemize}
	\item $l_{n}$ denotes the number of all lattices of size $n$; see \cite[A006966]{sloane},
	\item $c_{n}$ denotes the number of all congruence-uniform lattices of size $n$; see \cite[A292790]{sloane},
	\item $s_{n}$ denotes the number of spherical congruence-uniform lattices of size $n$; see \cite[A292852]{sloane}, and 
	\item $S_{n}$ denotes the number of all congruence-uniform lattices of size $n$ whose core label order is a lattice; see \cite[A292853]{sloane}.
\end{itemize}

\begin{table}
	\hspace*{-.5cm}\begin{tabular}{c||c|c|c|c|c|c|c|c|c|c|c|c|c|c}
		$n$ & $1$ & $2$ & $3$ & $4$ & $5$ & $6$ & $7$ & $8$ & $9$ & $10$ & $11$ & $12$ & $13$ & $14$\\
		\hline\hline
		$l_{n}$ & $1$ & $1$ & $1$ & $2$ & $5$ & $15$ & $53$ & $222$ & $1078$ & $5994$ & $37622$ & $262776$ & $2018305$ & $16873364$\\
		\hline
		$c_{n}$ & $1$ & $1$ & $1$ & $2$ & $4$ & $9$ & $22$ & $60$ & $174$ & $534$ & $1720$ & $5767$ & $20013$ & $71545$\\
		\hline
		$s_{n}$ & $1$ & $1$ & $0$ & $1$ & $1$ & $2$ & $3$ & $8$ & $17$ & $45$ & $123$ & $367$ & $1148$ & $3792$\\
		\hline
		$S_{n}$ & $1$ & $1$ & $0$ & $1$ & $1$ & $2$ & $3$ & $8$ & $16$ & $41$ & $107$ & $304$ & $891$ & $2735$\\
	\end{tabular}
	\caption{Numerology of congruence-uniform lattices.}
	\label{tab:enumeration}
\end{table}

\subsection{The Intersection Property}
	\label{sec:sip}
Let $\Lattice=(L,\leq)$ be a congruence-uniform lattice.  We say that $\Lattice$ has the \defn{intersection property} if for all $x,y\in L$ there exists some $z\in L$ with $\Shard(x)\cap\Shard(y)=\Shard(z)$.  It turns out that $\Lattice$ having the intersection property is equivalent to $\Alternate(\Lattice)$ being a meet-semilattice.

\begin{theorem}\label{thm:sip}
	The core label order of a congruence-uniform lattice $\Lattice$ is a meet-semilattice if and only if $\Lattice$ has the intersection property.
\end{theorem}
\begin{proof}
	If $\Lattice=(L,\leq)$ has the intersection property, then $\bigl\{\Shard(x)\mid x\in L\bigr\}$ is closed under intersections, which means that $\Alternate(\Lattice)$ is a meet-semilattice.  
	
	Conversely, suppose that $\Alternate(\Lattice)$ is a meet-semilattice, and let $x,y\in L$.  By assumption the meet $z$ of $x$ and $y$ in $\Alternate(\Lattice)$ exists, and by construction we have $\Psi(z)\subseteq\Psi(x)\cap\Psi(y)$.  On the other hand, however, if $j\in\Psi(x)\cap\Psi(y)$, then $j\in\ji(\Lattice)$ and we have $\Psi(j)=\{j\}$.  It follows then that $j\sqsubseteq x$ and $j\sqsubseteq y$, and therefore $j\sqsubseteq z$.  This implies $j\subseteq\Psi(z)$, from which follows that $\Psi(z)=\Psi(x)\cap\Psi(y)$.  Consequently, $\Lattice$ has the intersection property.
\end{proof}

The proof of the implication ``$\Alternate(\Lattice)$ is a meet-semilattice implies $\Lattice$ has the intersection property'' was suggested to us by a referee.  The proof of Theorem~\ref{thm:alternate_lattice} is now immediate.

\begin{proof}[Proof of Theorem~\ref{thm:alternate_lattice}]
	The first part is precisely Theorem~\ref{thm:sip}.  Lemma~\ref{lem:maximal_shard_spherical} states that $\Alternate(\Lattice)$ has a greatest element if and only if $\Lattice$ is spherical.  It is a classical lattice-theoretic result that a finite meet-semilattice with a greatest element is a lattice; see for instance \cite[Lemma~9-2.1]{reading16lattice}.
\end{proof}

Of course, we have just shifted the question when $\Alternate(\Lattice)$ is a meet-semilattice to the question when $\Lattice$ has the intersection property.  

It follows from \cite[Section~9-7.4]{reading16lattice} that a congruence-uniform lattice of regions has the intersection property.

We now prove that the intersection property is inherited to quotient lattices.

\begin{lemma}\label{lem:shard_sets_congruences}
	Let $\Lattice=(L,\leq)$ be a congruence-uniform lattice and let $\Theta\in\Con(\Lattice)$.  Let $\Sigma=\bigl\{j\in\ji(\Lattice)\mid (j_{*},j)\in\Theta\bigr\}$.  For $x\in L$ the set $\Shard_{\Lattice/\Theta}\bigl([x]_{\Theta}\bigr)$ is in bijection with $\Shard_{\Lattice}(x)\setminus\Sigma$.  
\end{lemma}
\begin{proof}
	The map 
	\begin{displaymath}
		f\colon\ji(\Lattice)\setminus\Sigma\to\ji(\Lattice/\Theta),\quad j\mapsto[j]_{\Theta}
	\end{displaymath}
	is by construction a bijection.  Let $x\in L$.  If $j\in\Shard_{\Lattice}(x)$, then we have $j=j_{\cg(u,v)}$ for some $x_{\downarrow}\leq u\lessdot v\leq x$.
	
	If $j\in\Sigma$, then we have $(u,v)\in\Theta$, and therefore $[u]_{\Theta}=[v]_{\Theta}$.  In particular we conclude $[j]_{\Theta}\notin\Shard_{\Lattice/\Theta}\bigl([x]_{\Theta}\bigr)$.  If $j\notin\Sigma$, then we have $[u]_{\Theta}\lessdot[v]_{\Theta}$, and we conclude $[j]_{\Theta}\in\Shard_{\Lattice/\Theta}\bigl([x]_{\Theta}\bigr)$.
	
	It follows that $f$ is the desired bijection from $\Shard_{\Lattice}(x)\setminus\Sigma$ to $\Shard_{\Lattice/\Theta}\bigl([x]_{\Theta}\bigr)$.
\end{proof}

\begin{proposition}\label{prop:spherical_quotients}
	Let $\Lattice$ be a spherical congruence-uniform lattice.  For every $\Theta\in\Con(\Lattice)$ the quotient lattice $\Lattice/\Theta$ is spherical, too.
\end{proposition}
\begin{proof}
	Let $\Sigma=\bigl\{j\in\ji(\Lattice)\mid (j_{*},j)\in\Theta\bigr\}$, and let $A=\Atoms(\Lattice)\cap\Sigma$, and let $B=\Atoms(\Lattice)\setminus A$.  Then we have 
	\begin{displaymath}
		\Atoms(\Lattice/\Theta)=\bigl\{[b]_{\Theta}\mid b\in B\bigr\}.  
	\end{displaymath}
	Let $x=\bigvee B$; the element $[x]_{\Theta}=\bigvee_{b\in B}{[b]_{\Theta}}$ is thus the join of all atoms in $\Lattice/\Theta$.  Moreover, since $\Lattice$ is spherical, Proposition~\ref{prop:atom_join_spherical} implies that $\grtst=\bigl(\bigvee A\bigr)\vee\bigl(\bigvee B\bigr)$.  For $a\in A$, we have by definition that $[a\vee x]_{\Theta}=[\least\vee x]_{\Theta}=[x]_{\Theta}$.  We thus obtain
	\begin{displaymath}
		[x]_{\Theta} = \Bigl[\bigl(\bigvee A\bigr)\vee x\Bigr]_{\Theta} = \Bigl[\bigl(\bigvee A\bigr)\vee \bigl(\bigvee B\bigr)\Bigr]_{\Theta} = [\grtst]_{\Theta}.
	\end{displaymath}
	Proposition~\ref{prop:atom_join_spherical} thus implies that $\Lattice/\Theta$ is spherical.
\end{proof}

\begin{proposition}\label{prop:intersection_property_quotients}
	Let $\Lattice$ be a congruence-uniform lattice with the intersection property.  For every $\Theta\in\Con(\Lattice)$ the quotient lattice $\Lattice/\Theta$ has the intersection property, too.
\end{proposition}
\begin{proof}
	Let $\Sigma=\bigl\{j\in\ji(\Lattice)\mid (j_{*},j)\in\Theta\bigr\}$, and let $f$ be the bijection from Lemma~\ref{lem:shard_sets_congruences}.  Consequently, for $x\in L$ we have $\Shard_{\Lattice/\Theta}\bigl([x]_{\Theta}\bigr)=f\bigl(\Shard_{\Lattice}(x)\setminus\Sigma)$.  
	
	Fix $x,y\in L$.  Since $\Lattice$ has the intersection property we can find $z\in L$ with $\Shard_{\Lattice}(z)=\Shard_{\Lattice}(x)\cap\Shard_{\Lattice}(y)$.  We then have 
	\begin{multline*}
		\Shard_{\Lattice/\Theta}\bigl([x]_{\Theta}\bigr)\cap\Shard_{\Lattice/\Theta}\bigl([y]_{\Theta}\bigr) = f\bigl(\Shard_{\Lattice}(x)\setminus\Sigma\bigr) \cap f\bigl(\Shard_{\Lattice}(y)\setminus\Sigma\bigr) =\\
			f\Bigl(\bigl(\Shard_{\Lattice}(x)\cap\Shard_{\Lattice}(y)\bigr)\setminus\Sigma\Bigr) = f\bigl(\Shard_{\Lattice}(z)\setminus\Sigma\bigr) = \Shard_{\Lattice/\Theta}\bigl([z]_{\Theta}\bigr).
	\end{multline*}
	It follows that $\Lattice/\Theta$ has the intersection property.
\end{proof}

We conclude this section with the proof of Theorem~\ref{thm:sip_congruences}.

\begin{proof}[Proof of Theorem~\ref{thm:sip_congruences}]
	It follows from Proposition~\ref{prop:congruence_uniform_pseudovariety} that $\Lattice/\Theta$ is congruence-uniform, and Propositions~\ref{prop:spherical_quotients} and \ref{prop:intersection_property_quotients} imply that $\Lattice/\Theta$ is spherical and has the intersection property.  The claim then follows from Theorem~\ref{thm:alternate_lattice}.
\end{proof}

\section{A New Characterization of Boolean Lattices}
	\label{sec:boolean_nexus}
In this section we attempt to give a conceptual interpretation of the set of core labels of some element $x$ of a congruence-uniform lattice.  Proposition~\ref{prop:canonical_shards} implies that $\Canonical(x)\subseteq\Shard(x)$ holds for all $x\in L$.  We show now that equality holds precisely in the case where $[x_{\downarrow},x]$ is a Boolean lattice.

\begin{proposition}\label{prop:canonical_shard_equality}
	Let $\Lattice=(L,\leq)$ be a congruence-uniform lattice, and let $x\in L$.  We have $\Canonical(x)=\Shard(x)$ if and only if $[x_{\downarrow},x]\cong\Boolean(k)$, where $k=\bigl\lvert\Canonical(x)\bigr\rvert$.
\end{proposition}
\begin{proof}
	Proposition~\ref{prop:congruence_uniform_pseudovariety} implies that intervals of congruence-uniform lattices are congruence uniform again.  We may thus assume that $x=\grtst$ and $x_{\downarrow}=\least$.  Moreover, let $k=\bigl\lvert\Atoms(\Lattice)\bigr\rvert$, which in view of Lemma~\ref{lem:representation_swap} means that $\bigl\lvert\Canonical(\grtst)\bigr\rvert=k$.
	
	By construction we find that $\Canonical(\grtst)=\Shard(\grtst)$ if and only if $\ji(\Lattice)=\Atoms(\Lattice)$, which is equivalent to $\Lattice$ being atomic.  The claim follows now from Theorem~\ref{thm:boolean_atomic_semidistributive}.
\end{proof}

Consequently, the size of the set $\Shard(x)\setminus\Canonical(x)$ tells us ``how far off'' the interval $[x_{\downarrow},x]$ is from a Boolean lattice.  In other words, $\bigl\lvert\Shard(x)\setminus\Canonical(x)\bigr\rvert$ is precisely the number of doublings that we need to ``undo'' in order to turn $[x_{\downarrow},x]$ into a Boolean lattice.  For a congruence-uniform lattice $\Lattice=(L,\leq)$ we may thus define the \defn{Boolean defect} of $\Lattice$ by
\begin{displaymath}
	\bdef(\Lattice) \defs \sum_{x\in L}{\bigl\lvert\Shard(x)\setminus\Canonical(x)\bigr\rvert}.
\end{displaymath}

The lattice in Figure~\ref{fig:cu_lattice} has Boolean defect $4$, and the lattice in Figure~\ref{fig:congruence_uniform_no_shard_lattice_cu} has Boolean defect $3$.

\begin{proposition}\label{prop:boolean_defect_zero}
	Let $\Lattice=(L,\leq)$ be a congruence-uniform lattice.  We have $\bdef(\Lattice)=0$ if and only if $[x_{\downarrow},x]$ is isomorphic to a Boolean lattice for all $x\in L$.
\end{proposition}
\begin{proof}
	This follows from Proposition~\ref{prop:canonical_shard_equality}.
\end{proof}

\begin{corollary}\label{cor:boolean_defect}
	A spherical congruence-uniform lattice $\Lattice$ has $\bdef(\Lattice)=0$ if and only if $\Lattice\cong\Boolean(n)$ for some $n\in\mathbb{N}$.
\end{corollary}
\begin{proof}
	This follows from the dual of Proposition~\ref{prop:atom_join_spherical} and Proposition~\ref{prop:boolean_defect_zero}.
\end{proof}

Perhaps the simplest example of a non-spherical congruence-uniform lattice with Boolean defect $0$ is a chain of length at least three.

If $\Lattice$ is join semidistributive, then for every $X\subseteq\Atoms(\Lattice)$ there exists an element $x\in L$ with $\Canonical(x)=X$.  Let us define the \defn{Boolean nexus} of $\Lattice$ by
\begin{displaymath}
	\Nexus(\Lattice) \defs \bigl\{x\in L\mid\Canonical(x)\subseteq\Atoms(\Lattice)\bigr\}.
\end{displaymath}

\begin{proposition}\label{prop:core_is_boolean}
	Let $\Lattice=(L,\leq)$ be a congruence-uniform lattice with $n$ atoms.  Then $\bigl(\Nexus(\Lattice),\leq\bigr)\cong\Boolean(n)$.  
\end{proposition}
\begin{proof}
	This follows from Lemma~\ref{lem:representation_swap}.
\end{proof}

\begin{corollary}\label{cor:core_shards}
	For $x\in\Nexus(\Lattice)$ we have $\Shard(x)=\bigl\{j\in\ji(\Lattice)\mid j\leq x\bigr\}$.
\end{corollary}
\begin{proof}
	Since $\Canonical(x)\subseteq\Atoms(\Lattice)$, Lemma~\ref{lem:representation_swap} implies that $x_{\downarrow}=\least$, and thus $\Shard(x)=\bigl\{j\in\ji(\Lattice)\mid j\leq x\bigr\}$.  
\end{proof}

\begin{proposition}\label{prop:core_is_shard_induced}
	If $\Lattice=(L,\leq)$ is a congruence-uniform lattice, then the poset $\bigl(\Nexus(\Lattice),\leq\bigr)$ is an induced subposet of $\Alternate(\Lattice)$.
\end{proposition}
\begin{proof}
	We need to show that for $x,y\in\Nexus(\Lattice)$ we have $x\leq y$ if and only if $\Shard(x)\subseteq\Shard(y)$.  Corollary~\ref{cor:shard_order_weakening} establishes one direction, and Corollary~\ref{cor:core_shards} implies the other.
\end{proof}

Let $\Lattice=(L,\leq)$ be a finite lattice, and let $C\subseteq L$ be a crosscut.  Recall that the \defn{crosscut complex} of $\Lattice$ (with respect to $C$) is the simplicial complex $\Cross(\Lattice;C)$ whose ground set is $C$ and whose faces are the subsets $B\subseteq C$ for which $\bigvee B$ or $\bigwedge B$ exists and belongs to $L\setminus\{\least,\grtst\}$.  (In particular, any subset of $C$ whose join is $\grtst$ and whose meet is $\least$ is not a face of $\Cross(\Lattice;C)$.)

If $\Lattice$ is congruence uniform and we choose $C=\Atoms(\Lattice)$, then the faces of $\Cross\bigl(\Lattice;\Atoms(\Lattice)\bigr)$ correspond to the elements of either $\Nexus(\Lattice)\setminus\{\least,\grtst\}$ (if $\Lattice$ is spherical) or $\Nexus(\Lattice)\setminus\{\least\}$ (if $\Lattice$ is not spherical) via the map $X\mapsto\bigvee X$.  In particular, if $\bigl\lvert\Atoms(\Lattice)\bigr\rvert=n$, then $\Cross\bigl(\Lattice;\Atoms(\Lattice)\bigr)$ is homotopy equvialent to the boundary of a $(n-1)$-simplex (if $\Lattice$ is spherical) or to an $(n-1)$-simplex (if $\Lattice$ is not spherical).  See also \cite[Section~2]{mcconville17crosscut}.

We conclude this section with the proof of Theorem~\ref{thm:boolean_alternate_order}, which states that the Boolean lattices are the only congruence-uniform lattices isomorphic to their core label order.  This property may therefore be taken as a new characterization of Boolean lattices.

\begin{proof}[Proof of Theorem~\ref{thm:boolean_alternate_order}]
	Let $\Lattice\cong\Boolean(n)$ for some $n\in\mathbb{N}$.  Proposition~\ref{prop:canonical_shard_equality} implies that $\Shard(X)=\Canonical(X)$ for every $X\subseteq[n]$.  For $X\subseteq[n]$ we have $X=\bigcup{\Canonical(X)}$, which yields $\Boolean(n)\cong\Alternate\bigl(\Boolean(n)\bigr)$.
	
	Conversely, let $\Lattice$ be such that $\Lattice\cong\Alternate(\Lattice)$.  Suppose that $\Lattice$ has $n$ atoms.  We conclude that $\Alternate(\Lattice)$ has $n$ atoms as well, and it is immediate from the definition that the atoms of $\Alternate(\Lattice)$ are precisely the join-irreducible elements of $\Lattice$.  We conclude that an element of $\Lattice$ is join irreducible if and only if it is an atom, which means precisely that $\Lattice$ is atomic.  Since $\Lattice$ is also semidistributive by Proposition~\ref{prop:congruence_uniform_semidistributive} we conclude from Theorem~\ref{thm:boolean_atomic_semidistributive} that $\Lattice\cong\Boolean(n)$.
\end{proof}

\section{Lattices of Biclosed Sets}
	\label{sec:biclosed}
Let us finally explain what biclosed sets are.  Let $S$ be a (finite) set, and let $\wp(S)$ denote the power set of $S$.  A \defn{closure operator} is a map $\cl\colon\wp(S)\to\wp(S)$ which is extensive, monotone, and idempotent, \ie which has the following three properties:
\begin{itemize}
	\item for $X\subseteq S$ we have $X\subseteq\cl(X)$;
	\item for $X,Y\subseteq S$ we have that $X\subseteq Y$ implies $\cl(X)\subseteq\cl(Y)$; and
	\item for $X\subseteq S$ we have $\cl\bigl(\cl(X)\bigr)=\cl(X)$.
\end{itemize}
A set $X\subseteq S$ is \defn{closed} if $\cl(X)=X$.  It is straightforward to verify that the family of closed sets with respect to $\cl$ is closed under intersection.  Therefore, the poset $\Bigl(\bigl\{\cl(X)\mid X\subseteq S\bigr\},\subseteq\Bigr)$ is in fact a lattice.  

We say that $X\subseteq S$ is \defn{biclosed} if both $X$ and $S\setminus X$ are closed.  Let $\Bic(S)$ denote the set of biclosed sets of $S$ with respect to $\cl$.  We are mainly interested in the cases where $\bigl(\Bic(S),\subseteq\bigr)$ is a congruence-uniform lattice.  The following example exhibits a spherical lattice of biclosed sets, whose core label order is not a lattice.

\begin{example}\label{ex:biclosed_no_shard_lattice}
	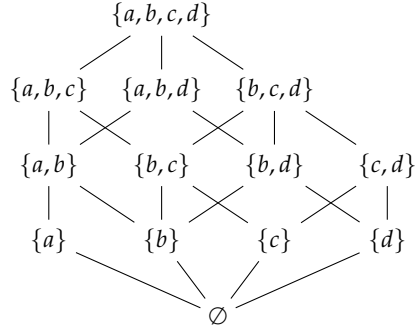
\begin{figure}
		\centering
		\begin{tikzpicture}\small
			\def\x{1.5};
			\def\y{1};
			\draw(2.5*\x,1*\y) node(n1){$\emptyset$};
			\draw(1*\x,2*\y) node(n2){$\{a\}$};
			\draw(2*\x,2*\y) node(n3){$\{b\}$};
			\draw(3*\x,2*\y) node(n4){$\{c\}$};
			\draw(4*\x,2*\y) node(n5){$\{d\}$};
			\draw(1*\x,3*\y) node(n6){$\{a,b\}$};
			\draw(2*\x,3*\y) node(n7){$\{b,c\}$};
			\draw(3*\x,3*\y) node(n8){$\{b,d\}$};
			\draw(4*\x,3*\y) node(n9){$\{c,d\}$};
			\draw(1*\x,4*\y) node(n10){$\{a,b,c\}$};
			\draw(2*\x,4*\y) node(n11){$\{a,b,d\}$};
			\draw(3*\x,4*\y) node(n12){$\{b,c,d\}$};
			\draw(2*\x,5*\y) node(n13){$\{a,b,c,d\}$};
			\draw(n1) -- (n2);
			\draw(n1) -- (n3);
			\draw(n1) -- (n4);
			\draw(n1) -- (n5);
			\draw(n2) -- (n6);
			\draw(n3) -- (n6);
			\draw(n3) -- (n7);
			\draw(n3) -- (n8);
			\draw(n4) -- (n7);
			\draw(n4) -- (n9);
			\draw(n5) -- (n8);
			\draw(n5) -- (n9);
			\draw(n6) -- (n10);
			\draw(n6) -- (n11);
			\draw(n7) -- (n10);
			\draw(n7) -- (n12);
			\draw(n8) -- (n11);
			\draw(n8) -- (n12);
			\draw(n9) -- (n12);
			\draw(n10) -- (n13);
			\draw(n11) -- (n13);
			\draw(n12) -- (n13);
		\end{tikzpicture}
		\caption{A lattice of closed sets.}
		\label{fig:closed_sets}
	\end{figure}

	Let $S=\{a,b,c,d\}$, and consider the closure operator given by the nontrivial assignments 
	\begin{displaymath}\begin{aligned}
		& \cl\bigl(\{a,c\}\bigr)=\{a,b,c\}, && \cl\bigl(\{a,d\}\bigr)=\{a,b,d\}, && \cl\bigl(\{a,c,d\}\bigr)=\{a,b,c,d\}.
	\end{aligned}\end{displaymath}
	For all other $X\subseteq S$ we have $\cl(X)=X$.  The lattice of closed sets of $\cl$ is shown in Figure~\ref{fig:closed_sets}.  It is quickly verified that this lattice is not meet semidistributive, and thus not congruence uniform.  (For instance, we have $\{b\}\wedge\{a\}=\emptyset=\{b\}\wedge\{c\}$, but $\{b\}\wedge\bigl(\{a\}\vee\{c\}\bigr)=\{b\}$.)  Its subposet of biclosed sets is shown in Figure~\ref{fig:biclosed_lattice}, and we can verify that it is indeed a spherical congruence-uniform lattice.  (It is isomorphic to the lattice in Figure~\ref{fig:congruence_uniform_no_shard_lattice_cu} doubled by the coatom $c_{2}$.)  The corresponding core label order, which is not a lattice, is shown in Figure~\ref{fig:biclosed_shard_order}.

	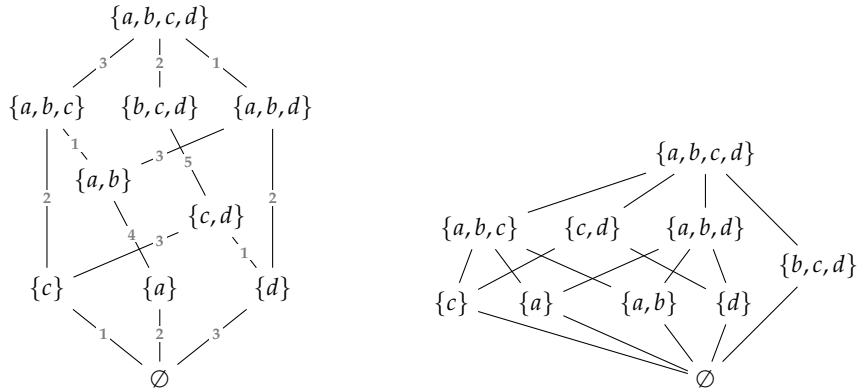
\begin{figure}
		\centering
		\begin{subfigure}[t]{.45\textwidth}
			\centering
			\begin{tikzpicture}\small
				\def\x{1.5};
				\def\y{1.2};
				\draw(2*\x,1*\y) node(n1){$\emptyset$};
				\draw(1*\x,2*\y) node(n2){$\{c\}$};
				\draw(2*\x,2*\y) node(n3){$\{a\}$};
				\draw(3*\x,2*\y) node(n4){$\{d\}$};
				\draw(1.5*\x,3.2*\y) node(n5){$\{a,b\}$};
				\draw(2.5*\x,2.8*\y) node(n6){$\{c,d\}$};
				\draw(1*\x,4*\y) node(n7){$\{a,b,c\}$};
				\draw(2*\x,4*\y) node(n8){$\{b,c,d\}$};
				\draw(3*\x,4*\y) node(n9){$\{a,b,d\}$};
				\draw(2*\x,5*\y) node(n10){$\{a,b,c,d\}$};
				\draw(n1) -- (n2) node[fill=white,inner sep=.9pt] at (1.5*\x,1.5*\y){\tiny\color{white!50!black}$\mathbf{1}$};
				\draw(n1) -- (n3) node[fill=white,inner sep=.9pt] at (2*\x,1.5*\y){\tiny\color{white!50!black}$\mathbf{2}$};
				\draw(n1) -- (n4) node[fill=white,inner sep=.9pt] at (2.5*\x,1.5*\y){\tiny\color{white!50!black}$\mathbf{3}$};
				\draw(n2) -- (n6) node[fill=white,inner sep=.9pt] at (2*\x,2.53*\y){\tiny\color{white!50!black}$\mathbf{3}$};
				\draw(n2) -- (n7) node[fill=white,inner sep=.9pt] at (1*\x,3*\y){\tiny\color{white!50!black}$\mathbf{2}$};
				\draw(n3) -- (n5) node[fill=white,inner sep=.9pt] at (1.75*\x,2.6*\y){\tiny\color{white!50!black}$\mathbf{4}$};
				\draw(n4) -- (n6) node[fill=white,inner sep=.9pt] at (2.75*\x,2.4*\y){\tiny\color{white!50!black}$\mathbf{1}$};
				\draw(n4) -- (n9) node[fill=white,inner sep=.9pt] at (3*\x,3*\y){\tiny\color{white!50!black}$\mathbf{2}$};
				\draw(n5) -- (n7) node[fill=white,inner sep=.9pt] at (1.25*\x,3.6*\y){\tiny\color{white!50!black}$\mathbf{1}$};
				\draw(n5) -- (n9) node[fill=white,inner sep=.9pt] at (2*\x,3.47*\y){\tiny\color{white!50!black}$\mathbf{3}$};
				\draw(n6) -- (n8) node[fill=white,inner sep=.9pt] at (2.25*\x,3.4*\y){\tiny\color{white!50!black}$\mathbf{5}$};
				\draw(n7) -- (n10) node[fill=white,inner sep=.9pt] at (1.5*\x,4.5*\y){\tiny\color{white!50!black}$\mathbf{3}$};
				\draw(n8) -- (n10) node[fill=white,inner sep=.9pt] at (2*\x,4.5*\y){\tiny\color{white!50!black}$\mathbf{2}$};
				\draw(n9) -- (n10) node[fill=white,inner sep=.9pt] at (2.5*\x,4.5*\y){\tiny\color{white!50!black}$\mathbf{1}$};
			\end{tikzpicture}
			\caption{A spherical concruence-uniform lattice of biclosed sets.}
			\label{fig:biclosed_lattice}
		\end{subfigure}
		\hspace*{.5cm}
		\begin{subfigure}[t]{.45\textwidth}
			\centering
			\begin{tikzpicture}\small
				\def\x{1.5};
				\def\y{1};
				\draw(3.5*\x,1*\y) node(n1){$\emptyset$};
				\draw(1.25*\x,2*\y) node(n2){$\{c\}$};
				\draw(2*\x,2*\y) node(n3){$\{a\}$};
				\draw(3*\x,2*\y) node(n4){$\{a,b\}$};
				\draw(3.75*\x,2*\y) node(n5){$\{d\}$};
				\draw(4.5*\x,2.5*\y) node(n6){$\{b,c,d\}$};
				\draw(1.5*\x,3*\y) node(n7){$\{a,b,c\}$};
				\draw(2.5*\x,3*\y) node(n8){$\{c,d\}$};
				\draw(3.5*\x,3*\y) node(n9){$\{a,b,d\}$};
				\draw(3.5*\x,4*\y) node(n10){$\{a,b,c,d\}$};
				\draw(n1) -- (n2);
				\draw(n1) -- (n3);
				\draw(n1) -- (n4);
				\draw(n1) -- (n5);
				\draw(n1) -- (n6);
				\draw(n2) -- (n7);
				\draw(n2) -- (n8);
				\draw(n3) -- (n7);
				\draw(n3) -- (n9);
				\draw(n4) -- (n7);
				\draw(n4) -- (n9);
				\draw(n5) -- (n8);
				\draw(n5) -- (n9);
				\draw(n6) -- (n10);
				\draw(n7) -- (n10);
				\draw(n8) -- (n10);
				\draw(n9) -- (n10);
			\end{tikzpicture}
			\caption{The core label order of the lattice in Figure~\ref{fig:biclosed_lattice}.}
			\label{fig:biclosed_shard_order}
		\end{subfigure}
		\caption{The poset of biclosed sets of the lattice in Figure~\ref{fig:closed_sets} is a spherical congruence-uniform lattice.  The corresponding core label order is not a lattice.}
		\label{fig:biclosed_sets}
	\end{figure}
\end{example}

Following T.~McConville in \cite[Section~2.5.1]{mcconville15biclosed} we say that $\mathcal{X}\subseteq\wp(S)$ is ordered by \defn{single-step inclusion} if for all $X,Y\in\mathcal{X}$ with $X\subsetneq Y$ there exists $x\in Y\setminus X$ such that $X\cup\{x\}\in\mathcal{X}$.  We quickly observe that the set $\Bic(S)$ from Example~\ref{ex:biclosed_no_shard_lattice} is not ordered by single-step inclusion.  (For instance, this lattice has the cover relation $\{c\}\lessdot\{a,b,c\}$.)  We are not aware of a spherical congruence-uniform lattice of biclosed sets which is ordered by single-step inclusion and whose core label order is not a lattice.

\begin{problem}\label{prob:single_step_no_lattice}
	Find a spherical congruence-uniform lattice of biclosed sets which is ordered by single-step inclusion and whose core label order is not a lattice.
\end{problem}

\subsection*{Acknowledgments}

I thank Al Garver, Thomas McConville and Nathan Reading for interesting discussions on the topic, in particular for suggesting the name ``core label order''.  Moreover, I would like to express my gratitude to an anonymous referee for many valuable suggestions that improved both content and exposition of this paper.

\bibliography{../../literature}

\end{document}